\documentclass[11pt, twoside, a4paper]{article}

\usepackage[utf8]{inputenc}
\usepackage{amsmath, amssymb, amsthm}            
\usepackage{amstext, amsfonts, a4}
\usepackage[english]{babel}
\usepackage{xcolor}
\usepackage{bbm}

\usepackage[margin=1.9cm]{geometry}
\usepackage{dsfont}

\usepackage[hidelinks]{hyperref}
\usepackage{soul}
\setstcolor{red}
\usepackage{cancel}

\usepackage[mathlines, pagewise]{lineno}

\numberwithin{equation}{section}

\DeclareMathOperator*{\wtend}{\rightharpoonup}

\theoremstyle{plain}
	\newtheorem{theorem}{Theorem}[section]
	\newtheorem{proposition}[theorem]{Proposition}    
	\newtheorem{lemma}[theorem]{Lemma}

	\newtheorem{property}[theorem]{Property}
	
\theoremstyle{definition}
	\newtheorem{definition}[theorem]{Definition}

 \def\NN{{\mathbb N}}  
  \def\RR{{\mathbb R}}

 \def\ZZ{{\mathbb Z}} 

 \def\cB{{\mathcal B}} \def\cC{{\mathcal C}} 
\def\cD{{\mathcal D}}  \def\cF{{\mathcal F}} 
  \def\cI{{\mathcal I}} 
  \def\cL{{\mathcal L}} 
 \def\cN{{\mathcal N}}  
   
\def\cS{{\mathcal S}}

\newcommand{\eps}{\varepsilon}
\newcommand{\R}{\mathbb{R}} 

\newcommand{\dd}{\mathrm{d}}

\newcommand{\der}[2]{\frac{\dd #1}{\dd #2}}

\newcommand{\tend}[2]{\underset{#1\to #2}{\longrightarrow} }

\newcommand{\mc}{\mathcal}
\newcommand{\what}{\widehat}

\newcommand{\dx}{ \, {\rm d} x}
\newcommand{\dt}{ \, {\rm d} t}

\DeclareMathOperator*{\D}{\rm{div}}

\title{Existence and Uniqueness for the SQG Vortex-Wave System when the Vorticity is Constant near the Point-Vortex}
\author{
Dimitri Cobb\footnote{Universität Bonn, Hausdorff Center for Mathematics (HCM), Endenicher Allee 62, D-53115 Bonn, Deutschland.}, 
Martin Donati\footnote{Université Grenoble Alpes, Institut Fourier, 100 rue des Mathématiques, 38610 Gières, France.} 
and Ludovic Godard-Cadillac\footnote{Université de Bordeaux, Bordeaux INP, Institut de Mathématiques de Bordeaux (IMB), UMR 5251, F-33400 Talence, France.}}
\date{\today}

\begin{document}

\maketitle

\begin{abstract}
This article studies the vortex-wave system for the Surface Quasi-Geostrophic equation with parameter $0<s<1$. We obtain local existence of classical solutions in $H^4$ under the standard ``plateau hypothesis'', $H^2$-stability of the solutions, and a blow-up criterion. 
In the sub-critical case $s>1/2$ we establish global existence of weak solutions. For the critical case $s=1/2$, we introduced a weaker notion of solution ($V$-weak solutions) to give a meaning to the equation and prove global existence. 

\end{abstract}

\section*{Introduction}

The purpose of this work is to study the vortex-wave system associated to \eqref{SQG}, the surface quasi-geostrophic equation. Equation \eqref{SQG} is a partial differential equation set on the plane that is widely used to model atmospheric flows outside the tropical region. It is given by
\begin{equation}\tag{SQG}\label{SQG}
    \frac{\partial}{\partial t}\theta+v\cdot\nabla\theta=0,\qquad\quad\text{with }\quad v=-\nabla^\perp(-\Delta)^{-s}\,\theta,
\end{equation}
where $0<s<1$ and where $\perp$ designates the counter-clockwise rotation of angle $\pi/2$ in the plane. 
In this equation, $\theta:[0,T)\times\RR^2\to\RR$ is the potential temperature of the fluid, that will be referred throughout this article as the \emph{active scalar}, and $v:[0,T)\times\RR^2\to\RR^2$ is the velocity of the fluid.
For the physical content of (SQG), we recommend the two reference books~\cite{Pedlowsky_1987, Vallis_2006}. 
The second equation in~\eqref{SQG} is called the \emph{quasi-geostrophic Biot-Savart law} in reference to the 2D Euler equation written in term of the vorticity $\omega$:
\begin{equation}\tag{Euler}\label{Euler}
    \frac{\partial}{\partial t}\omega+v\cdot\nabla\omega=0,\qquad\quad\text{with }\quad v=-\nabla^\perp(-\Delta)^{-1}\,\omega.
\end{equation}
The 2D Euler equation can be seen as the limit case in~\eqref{SQG} where $s\to 1^-$. 
The similarity of these two equations gave raise to a wide range of studies trying to investigate whether the special solutions known for 2D Euler have their equivalent for \eqref{SQG}. 
The most studied cases are traveling and co-rotating solutions, either from bifurcation arguments~\cite{Ao_Davila_DelPino_Musso_Wei_2021, Garcia_2021, Hassainia_Wheeler_2022, Hmidi_Mateu_2017} or variationnal arguments~\cite{Cao_Qin_Zhan_Zou_2021, Godard-Cadillac_Gravejat_Smets_2020, Gravejat_Smets_2019}.

However close these two equations may seem, the theory for the Cauchy problem associated to the evolution equation is very different. 
Existence and uniqueness globally in time for the 2D Euler equation is well-known since the works by Youdovich if the initial vorticity lays in $L^\infty$. 
In this case, the Euler Biot-Savart law gives that the velocity field is log-Lipschitz, so that the Osgood lemma applies (see~\cite{Youdovich_1963, Youdovich_1995}). 
For the~\eqref{SQG} equation, the Biot-Savart law is more singular and such an argument does not apply. 
The Cauchy theory for \eqref{SQG} remains to this day an outstanding problem. 
The global existence is only known for weak solutions~\cite{Marchand_2008}, which are non-unique in some cases when the regularity is too low~\cite{Buckmaster_Shkoller_Vicol_2019}. 
Numerical simulations~\cite{Scott_Dritschel_2014} suggests that, unlike the 2D Euler equations, the $L^\infty$ setting does not give raise to global in time strong solutions. 
Existence and uniqueness on a short interval of time for \eqref{SQG} has been proved for the sub-critical and critical case ($s\geq 1/2$) by~\cite{Constantin_Majda_Tabak_1994} and the super-critical case ($0<s<1/2$) has been done by~\cite{Chae_Constantin_Cordoba_Gancedo_Wu_2012} in the $H^k$ setting with $k\geq 4$. 
The study of the Cauchy problem for these equations is very active, trying to exhibit finite time singularity or to prove global existence in a well-chosen class of regularity (see for instance~\cite{Chae_Jeong_Oh_2023, Cordoba_MartinezZoroa_2022, Garcia_GomezSerrano_2022, Hu_Kukavika_Ziane_2015, Jeong_Kim_2024, Jolly_Kumar_Martinez_2022} and references therein).

Nevertheless, one feature of the 2D Euler equation that is preserved by \eqref{SQG} is the point-vortex system.
This consists in the assumption that at the initial time the active scalar is sharply concentrated around a finite number of points and therefore writes at the limit as a weighted sum of Dirac masses.
This model should be understood as a mathematical formalization of the intuitive notion of ``\emph{centers of whirlpools}''. 
These Dirac masses then evolve in time in virtue of the transport equations respectively~\eqref{SQG} and~\eqref{Euler}. Although in the SQG equations the active scalar is not the vorticity of the fluid (its curl), by analogy with the Euler case we continue to call a Dirac mass of the active scalar a \emph{point-vortex}.
This system is well-posed on a given interval of time as long as the distances between vortices remain positive (absence of collapse).
Global well-posedness for almost every initial datum has been obtained for the Euler case in the plane by~\cite{Marchioro_Pulvirenti_1993}, and in bounded smooth domains by~\cite{Donati_2022}. 
The extension of this result to the quasi-geostrophic point-vortex system is done by~\cite{Geldhauser_Romito_2020, Godard-Cadillac_vortex_2022}.

One natural question then lays in the possibility to study a mixed system involving both a regular part for the active-scalar (driven by the PDE) coupled with a singular part made of point-vortices.
This coupled system was introduced for the 2D Euler equation independently by~\cite{Marchioro_Pulvirenti_1991,Starovoitov_1994_Uniqueness} and called the ``\emph{vortex-wave system}''.
Using a Lagrangian formulation of the vortex-wave system, the existence of solutions for this system was proved in presence of a single point-vortex. 
Uniqueness for this system has been achieved by~\cite{Marchioro_Pulvirenti_1991} and~\cite{Lacave_Miot_2009} respectively for Lagrangian and Eulerian solutions, under the assumption that the point-vortex lays in a region where the vorticity is constant (plateau assumption).
It is worth mentioning other works related to the vortex-wave system such as~\cite{ Bjorland_2011, Cao_Wang_2020,Crippa_LopesFilho_Miot_NussenzveigLopes_2016, Glass_Lacave_Sueur_2014,  Glass_Lacave_Sueur_2016, Lacave_Miot_preprint, Nguyen_Nguyen_2019, Rosenzweig_2020, Varholm_Wahlen_Walsh_2020}.

The main goal of the present paper is to study the vortex-wave system for the Surface Quasi-Geostrophic equation.
Our first result is the local existence and uniqueness of a classical solution for the SQG vortex-wave system under the same ``plateau assumption'' as in~\cite{Lacave_Miot_2009}.
Unlike the Euler case, it is not clear whether this result can be made global since the velocity field is less regular due to the singularity of the Biot-Savart law kernel. 
We cannot obtain strong solution with $\theta$ being only in $L^1\cap L^\infty$ so that we make a much stronger assumption on the initial datum:   $\theta_0\in H^4(\RR^2).$
Such a regularity allows us to proceed to energy estimates relying on the commutator structure of the equation. 
While the energy estimates are comparable to the ones of \cite{Chae_Constantin_Cordoba_Gancedo_Wu_2012}, new difficulties arise due to the presence of the point-vortex.
Passing to the limit in a regularized system gives the existence of the solution on a bounded interval of time.
We were also able to obtain a uniqueness result with a proof that relies on the point-vortex staying in its initial ``plateau'', combined with a precise $H^2$-stability estimate.
We also provide a directional blow-up criterion for the case where the point-vortex reaches the boundary of its ``plateau'' in finite time.
In our work, we mainly focused on the novelties induced by the presence of the point-vortex but we believe it is possible to improve the regularity using some of the refined techniques developed in the literature on SQG (see \cite{Hu_Kukavika_Ziane_2015, Jolly_Kumar_Martinez_2022, Jolly_Kumar_Martinez_2021} for details).

The second part of the article studies weak solutions (in $L^1\cap L^\infty$) for the SQG vortex-wave system. 
Since the Youdovitch theory does not apply for SQG, we cannot expect global solutions in $L^1\cap L^\infty$ to be unique with characteristics that are well-defined.
In the case $s>1/2$ the Biot-Savart law has good regularizing properties, and we can prove global existence.
The critical case $s=1/2$ is more difficult since it is no longer possible to obtain regularity on the velocity field. The Biot-Savart law is here a $0$-order pseudo-differential operator.
The first consequence is that one must relax the definition of the point-vortex system to apply the Peano theorem.
Furthermore, the main difficulty lays in the definition of the velocity field generated by the point-vortex when acting on the smooth part of the active scalar. Indeed, the case $s=1/2$ implies that the kernel of the Biot-Savart law is no longer $L^1_{\rm loc}.$
To overcome this problem, we introduce a weaker notion of solution, that we called $V$-weak solutions, adapted to the study of singularities localized on a set of small Hausdorff dimension.
We show the existence of a global $V$-weak solution and study its properties. 
We show that this weaker notion of solutions shares with the standard notion of weak solutions several important properties. In particular, we prove that if the $V$-weak solution is $\cC^1$, then it is a classical solution.

\section{Presentation of the Problem and Main Result}

\subsection{The Vortex-Wave System for SQG}

\subsubsection{Singular vorticity: The point-vortex system}

We first define for $0<s\leq 1$ the function:
\begin{equation}\label{def:K_s}
   \forall\;x\in\RR^2\setminus\{0\},\qquad K_s(x):=c_s\frac{x^\perp}{|x|^{4-2s}},\qquad\text{with }\;c_s:=\frac{(1-s)\Gamma(1-s)}{2^{2s-1}\pi\Gamma(s)},
\end{equation}
which is the kernel of the Biot-Savart operator $-\nabla^\perp (-\Delta)^{-s}$ associated to the Green's function of $(-\Delta)^{s}$. Remark that $c_s\to\frac{1}{2\pi}=:c_1$ as $s\to1^-$. The case $s=1$ corresponds to the case of~\eqref{Euler} while $0<s<1$ corresponds to~\eqref{SQG}.
The point-vortex system in the plane is given by the following equations. If $z_1,\ldots,z_N$ are the positions of the point-vortices in the plane and $a_1,\ldots,a_N$ their intensities, then the evolution of the $i^{th}$ vortex with $i=1,\dots, N$ is given by:
\begin{equation}\label{eq:Point Vortex alpha}
\der{}{t} z_i(t) = c_s\sum_{\substack{j=1 \\ j\neq i}}^N  a_j\frac{(z_i(t)-z_j(t))^\perp}{|z_i(t)-z_j(t)|^{4-2s}}.
\end{equation}
Indeed, if one plugs an active scalar of the form $\theta(0,x)=\sum_{j=1}^Na_j \delta_{z_j}$ in~\eqref{SQG} or~\eqref{Euler} then the associated velocity field writes, using~\eqref{def:K_s}:
\begin{equation*}
    v(0,x)\;:=\; c_s\sum_{j=1}^Na_j\frac{(x-z_j)^\perp}{|x-z_j|^{4-2s}}.
\end{equation*}
Since the active scalar is simply transported by the velocity of the fluid, we indeed obtain~\eqref{eq:Point Vortex alpha} under the natural assumption that a point-vortex does not self-interact.

For this article, we mainly focus on the vortex-wave system in the case of a single vortex moving in time $t\mapsto z(t)\in\RR^2$.
The case of several point-vortices is similar as long as the distances between the point-vortices remain positive on the studied interval of time, see Section \ref{sec:multi_points}.
To give a proper definition of the vortex-wave system, the authors in~\cite{Marchioro_Pulvirenti_1991} suggest a Lagrangian point-of-view while in~\cite{Lacave_Miot_2009} they work with an Eulerian point-of-view. In this work we follow the second point-of-view. 

\subsubsection{Local Existence and Uniqueness for the SQG Vortex-Wave}

We prove the following well-posedness result for the vortex-wave system:
\begin{theorem}\label{thrm:strong_solutions}
    Let $0 < s < 1$. Consider $\theta_0 \in H^k$ an initial datum for some integer $k \geq 4$ and $z_0 \in \R^2$ an initial point-vortex position. Assume that $\theta_0$ is constant on some neighborhood of $z_0$ (plateau assumption). Then the following assertions hold true.
    \begin{itemize}
    
        \item[$(i)$] Local existence and uniqueness: there exists time $T > 0$ such that the PDE system
            \begin{equation}\label{eq:StrongSystem}
                \begin{cases}
                    \partial_t \theta + (v + H) \cdot \nabla \theta = 0, \vspace{1mm}\\
                    \displaystyle \der{}{t} z(t) = v \big(t, z(t) \big),\vspace{1mm}\\
                    v = - \nabla^\perp (- \Delta)^{-s} \theta = K_s\star\theta, \vspace{1mm}\\
                    H(t,x) = K_s \big( x - z(t) \big),
                \end{cases}
            \end{equation}
            with initial datum respectively $\theta_0$ and $z_0$,
            has a unique solution $\theta \in L^\infty([0, T) ; H^k(\R^2)) \cap C^1([0, T) \times \R^2)$ and $z\in C^1([0, T) ; \R^2)$.
    
        \item[$(ii)$] Short time $H^{2}$-stability: 
            let $(\theta_1,z_1)$ and $(\theta_2,z_2)$ be two solutions in $L^\infty([0,T],H^k(\R^2))\times C^1([0,T];\RR^2)$ of \eqref{eq:StrongSystem}, with $\theta_1(0)$ and $\theta_2(0)$ locally constant around $z_1$ and $z_2$ respectively. Then there exists $\rho > 0$ and $T_0 >0$ depending only on $\theta_1(0)$ and $\theta_2(0)$  such that if
            \begin{equation*}
                |z_2(0) - z_1(0) | < \rho,
            \end{equation*}
            then, for every $t \in [0,T_0]$, and for every $\ell \in \NN$ such that $2 \le \ell \le k-2$,
            \begin{equation}\label{ieq:Stability}
                \| \theta_2(t,\cdot) - \theta_1(t,\cdot) \|_{H^{\ell}} + |z_2(t)-z_1(t)| \lesssim \| \theta_2(0,\cdot) - \theta_1(0,\cdot) \|_{H^{\ell}} + |z_2(0)-z_1(0)|.
            \end{equation}
            
        \item[$(iii)$] Blow-up criterion: let $(\theta,z)$ be a solution of~\eqref{eq:StrongSystem} on an interval $[0,T^*)$ with $T^* <\infty$. Let
            \begin{equation}\label{def:R(t)}
                R(t) := \sup \big\{ r \ge 0 \; , \; \nabla\theta(t) \equiv 0 \text{ on } B(z(t),r) \big\}
            \end{equation}
            and recall that the plateau hypothesis states that $R(0)>0$. Assume that $R(t) \tend{t}{T^*} 0$. Then
            \begin{equation*}
                \int_0^{T^\ast} \|\theta(t)\|_{H^{3-2s}} \dd t = +\infty.
            \end{equation*}
    \end{itemize}

\vspace{0.2cm}
\end{theorem}

The different statements in Theorem~\ref{thrm:strong_solutions} are proved throughout the article. The existence and regularity of strong solutions in point \textit{(i)} is the main concern of Section~\ref{sec:existence_strong}, and Theorem~\ref{t:StrongExistence} in particular. For point \textit{(ii)}, uniqueness and the proof of~\ref{ieq:Stability} comes from Section~\ref{sec:Uniqueness} and Proposition \ref{prop:stabilité} therein. Finally, the results of \textit{(iii)} concerning a blow-up criterion is proved in Section~\ref{sec:blow-up}, see also Theorem~\ref{thrm:blow-up} below.

\medskip

Let us formulate a few remarks about Theorem~\ref{thrm:strong_solutions}.

\medskip

In~\eqref{eq:StrongSystem}, the equation $v = -\nabla^\perp (- \Delta)^{-s} \theta$ is well-defined as a convolution with $K_s$ only when $s> 1/2$. Yet it is not always clear whether this convolution is well-defined. When $s\leq 1/2$ we use Fourier multipliers to give a meaning to the velocity $v$ (see Proposition~\ref{p:FourierMultiplierKS} hereafter). Note also that for a fixed time $t\in[0,T)$, the product $H \theta$ is a well-defined and measurable function taking finite values for every $x\neq z(t)$. It is also well-defined on all $\RR^2$ when seen as a principal value distribution, because $\theta(t) \in H^k(\R^2) \subset H^4(\R^2)$ and $H(t) \in \dot{B}^{2s - 1}_{1, \infty}(\R^2) \subset H^{-2}(\R^2)$.
The convolution against the Biot-savart kernel $K_s$ is taken with respect to the space variable $x\in\RR^2.$ 
As a convention throughout this article, the gradient operators $\nabla$ and $\nabla^\perp$, convolutions and fractional Laplace operators $(-\Delta)^s$ and $(-\Delta)^{-s}$ are always taken with respect to the space variable $x\in\RR^2.$

Concerning the regularity of strong solutions, in this work we limit ourselves to $H^4$ solutions for simplicity reasons and focus on what is new in presence of a point-vortex. However, we believe that both existence and uniqueness results can be extended to lower levels of regularity on $\theta$, according to \cite{Hu_Kukavika_Ziane_2015} and \cite{Jolly_Kumar_Martinez_2022}. For the stability result, the hypothesis $2 \le \ell \le k-2$ is always satisfied at least by $\ell =2$ since $k\ge 4$. In the case $s \ge 1/2$ we also obtain a $L^2$-stability result (see Section~\ref{sec:Uniqueness} for rigorous statement and proof). This $L^2$-stability fails in the case $s<1/2$ for which we need at least to control the $H^2$ norm of $\theta$.

The system~\eqref{eq:StrongSystem} can blow-up in two different ways: either the active scalar could lose regularity on its own as in the usual SQG equations, or the point-vortex could collide with the boundary of the constant part of the active scalar, making the system singular. Theorem~\ref{thrm:strong_solutions}-$(iii)$ investigates this second type of blow-up. 
In Section~\ref{sec:blow-up}, we actually obtain a sharper and directional criterion, which is the following Theorem~\ref{thrm:blow-up}, from which we deduce Theorem~\ref{thrm:strong_solutions}-$(iii)$. More precisely, we prove:

\begin{theorem}\label{thrm:blow-up}
    Let $(\theta,z)$ be a strong solution of~\eqref{eq:StrongSystem} as given by Theorem~\ref{thrm:strong_solutions} on an interval $[0,T^*)$, with $T^* < \infty$. Let
    \begin{equation*}
    N(t) := \max_{\{x \in \R^2, \, |x-z(t)| = R(t)\}} \frac{-\big(x-z(t)\big)\cdot\big(v(x)-v(z(t))\big)}{R^2(t)(1-\ln R(t))} \,.
    \end{equation*}
    Then,
    \begin{equation*}
        R(t) \tend{t}{T^*} 0 \qquad \Longleftrightarrow \qquad \int_0^{T^*} N(t) \dd t = +\infty \,.
    \end{equation*}
\end{theorem}
The quantity $t \mapsto N(t)$ is a weaker version of the log-Lipschitz norm of $v$ that only takes into account the capacity of $v$ to bring fluid particles from the outside of $B(z(t),R(t))$ towards the point-vortex.

The proof of Theorem~\ref{thrm:blow-up} is done along with Theorem~\ref{thrm:strong_solutions}-$(iii)$ in Section~\ref{sec:blow-up}.

\subsubsection{Global existence of weak solution in the sub-critical case}
In the sub-critical case (namely $s>1/2$) the global existence of weak solutions is ensured due to the regularizing properties of the $\nabla^\perp(-\Delta)^{-s}$ operator. More precisely, we prove the following.

\begin{theorem}[Global weak solutions if $s>1/2$]~\label{thrm:global weak sub-critical}
    We assume that $s > 1/2$. Let $\theta_0 \in L^1 \cap L^\infty(\RR^2)$ and $z_0 \in \R^2$. Then there exists $\theta \in L^\infty(\R_+ ; L^1 \cap L^\infty)$ and $z \in C^0(\R_+ ; \R^2)$ such that $(\theta, z)$ is a weak solution of the vortex-wave system with initial datum $(\theta_0, z_0)$. More precisely, for any test function $\psi \in \mc D ([0, T) \times \R^2)$, we have
\begin{equation}\label{eq:weak formulation}
    \int_0^T\int_{\RR^2}\theta(t,x)\,\bigg(\frac{\partial\psi}{\partial t}(t,x)+(v+H)\cdot\nabla\psi(t,x)\bigg)\,\dd x\,\dd t+\int_{\RR^2}\theta_0(x)\psi(0,x)\,\dd x=0,
\end{equation}
where $v$ and $H$ are given by~\eqref{eq:StrongSystem} and where $z$ is a solution of the integral equation
\begin{equation}\label{eq:weak formulation 2}
    z(t) = z_0 + \int_0^t v \big( \tau, z(\tau) \big) {\rm d} \tau.
\end{equation}
Moreover, this solution satisfies $\|\theta(t,\cdot)\|_{L^p}\leq\|\theta_0\|_{L^p}$ for all $p\in[1,+\infty].$ Finally, the function $\theta$ is continuous in time with values in $L^1\cap L^\infty(\RR^2)$ equipped with the weak-$\star$ topology.
\end{theorem}
The proof is given in Section~\ref{sec:proof-subcritical}.

\subsection{Global Existence of Weak Solutions in the Critical Case}
From the result given by Theorem~\ref{thrm:global weak sub-critical}, one naturally wonders what happens in the critical case where $s=1/2$.

\subsubsection{Duhamel Formula for the Quasi-Point-Vortex System}
One important difference in the case $s=1/2$ is that the function $K_s$ ceases to be locally $L^1$. 
Moreover, the convolution with $K_s$ is now a Fourier multiplier of order $0$ since such convolution correspond to the differential operator $-\nabla^\perp(-\Delta)^{-\frac{1}{2}}$.
As a consequence, there is no hope for the velocity $v:=K_s\star\theta$ to be a continuous function with $\theta$ being only in $L^1\cap L^\infty$. 
To give a meaning to the vortex-wave system when $s=1/2,$ we have to weaken the notion of point-vortex to obtain the regularity needed to define the flow. 
For that purpose, we rewrite the equation on the evolution of the point-vortex~\eqref{eq:Euler Vortex Wave weak 2} as follows:
\begin{equation}\label{eq:not relaxed}
    z(t)=z_0+\int_0^t\int_{\RR^2}v\big(t',x\big)\,\delta_0\big(x-z(t'))\,\dd x\,\dd t',
\end{equation}
where $\delta_0$ is the centered Dirac mass.
In the view of such a formulation we introduce the notion of \emph{quasi point-vortex}, which is a relaxation of the equation above:
\begin{equation}\label{eq:relaxed}
    z(t)=z_0+\int_0^t\int_{\RR^2}v\big(t',x\big)\,\chi\big(x-z(t'))\,\dd x\,\dd t',
\end{equation}
where $\chi$ is a fixed $L^1$ function such that $\int\chi=1$. 
This function $\chi$ must be understood as an approximation of the Dirac mass.
In other words, the computation of the velocity for the point-vortex is not the simple evaluation of the velocity field $v$ at one point $z$ (which is ill-defined) but the computation of an average of $v$ on a ``small'' neighborhood of $z$.
The main weakness of this approach is that the solution depends on the choice of $\chi\in L^1$ and it is unclear whether one can pass to the limit $\chi\to\delta_0$.
Yet this relaxation ~\eqref{eq:relaxed} is the most simple and reasonable way to give a meaning to~\eqref{eq:not relaxed}.

\subsubsection{A Remark on the Structure of the Non-Linearity}
Another problem arising in the critical case $s=1/2$ is to give a meaning for almost every $t\geq0$ to the following integrals appearing in~\eqref{eq:weak formulation}:
\begin{equation}\label{eq:Integrals not defined}
    \int_{\RR^2}\theta(t,x)\,v(t,x)\cdot\nabla\psi(t,x)\,\dd x\qquad\text{and}\qquad\int_{\RR^2}\theta(t,x)\,H(t,x)\cdot\nabla\psi(t,x)\,\dd x.
\end{equation}

Concerning the first one, we simply follow the idea developed in~\cite{Marchand_2008} which consists in using the commutator structure of this term to regain some regularity on $\theta$.
This idea is reminiscent from previous works on SQG.

\begin{lemma}[Commutator formulation]\label{lem:Commutator formulation}
    Let $\theta\in\cC^1$ and $\varphi\in\cD(\RR^2)$. We have:
    \begin{equation}\label{eq:commutator formulation}
        \int_{\RR^2}\theta(x)\,K_s\!\star\theta(x)\cdot\nabla\varphi(x)\,\dd x=\frac{1}{2}\int_{\RR^2}K_s\!\star\theta(x)\cdot\Big[(-\Delta)^s,\nabla\varphi\Big](-\Delta)^{-s}\theta(x)\,\dd x,
    \end{equation}
    where $[A,B]:=AB-BA$ is the commutator between $A$ and $B$.
\end{lemma}
This lemma is proved later on (Section~\ref{sec:commutator}). 
It is proved in~\cite{Marchand_2008}, using this commutator formulation, that this integral makes sense in $L^p$ for $4/3<p<2$ (see Lemma~\ref{lem:L4/3 regularity} hereafter).

\subsubsection{A Weaker Notion of Weak Solution}

The remaining difficulty is the second integral in~\eqref{eq:Integrals not defined}, which is more difficult to treat since $H= -\nabla^\perp (-\Delta)^{-s}\delta_z$, which is more singular than $v:=-\nabla^\perp(-\Delta)^{-s}\theta.$ 
One possible idea would have been to regularize $\delta_0(x-z(t))$ in the same spirit than~\eqref{eq:relaxed}.
Nevertheless, such an approach would have made loose the information on the deformation of $\theta$ in the neighborhood of the singularity of a point-vortex.

To pass-by this difficulty, we choose to introduce a weaker notion of solution that is weaker than the standard notion of \emph{weak solution} exploiting the fact that the singularity in the equation lays in a small set : a single point moving in time. This idea is inspired by the work of Schochet~\cite{Schochet_1995} on Euler point-vortices.
We first recall that, in a general formulation, the standard notion of weak solution writes: 
\begin{equation}\label{eq:weak standard}
    u\in \cD'(\Omega;\RR^d)\quad\text{ is a weak solution iff: }\qquad\forall\;\phi\in\cD(\Omega;\RR^d),\qquad\cF(u,\phi)=0
\end{equation}
where $\cF:\cD'\times\cD\to\RR$ is the function that represent the weak formulation of the studied equation
(the function $u\mapsto\cF(u,\phi)$ may be well-defined only on a subspace of $\cD'$).

\begin{definition}[$V$-weak solutions]\label{defi:V weak}
    Let $\Omega$ be a domain and $V\subseteq\cD(\Omega;\RR^d).$ Then
    \begin{equation*}
    u\in \cD'(\Omega;\RR^d)\quad\text{ is a }V\text{-weak solution to~\eqref{eq:weak standard} iff: }\quad\forall\;\phi\in V,\quad\cF(u,\phi)=0.
\end{equation*}
\end{definition}
In other words, we are not able to have the weak formulation~\eqref{eq:weak standard} for all test functions $\phi\in\cD(\Omega;\RR^d)$ but only for functions in a subset $V$.
To have this notion of $V$-weak solution be relevant, this set $V$ must be the biggest possible.
Whenever we manipulate it, we systematically study how big the set $V\subseteq\cD(\Omega;\RR^d)$ can be to ensure that we remain close to the standard notion of weak solution.
Note that when $u$ and $\cF$ are smooth enough, the two notion coincide whenever $V$ is dense in $\cD(\Omega;\RR^d)$.
To measure how much the set $V$ coincide with $\cD(\Omega;\RR^d)$, we introduce the notion of $\Gamma$-coincidences for sets of functions:
\begin{definition}[$\Gamma$-coincidence]
    Let $V$ and $W$ two sets of functions $\Omega\to\RR^d$ and let $\Gamma\subseteq\Omega$. We say that the two sets $V$ and $W$ are $\Gamma$-\textit{coincident} if they coincide except on $\Gamma$. More precisely, for all $\Omega_0\subseteq\Omega$ such that the distance between $\Omega_0$ and $\Gamma$ is positive:
    \begin{equation*}
        \forall\,f\in V,\;\;\exists\;g\in W,\quad f_{\big|_{\Omega_0}}\equiv g_{\big|_{\Omega_0}}\qquad\text{and symmetrically :}\qquad\forall\;g\in W,\;\;\exists\;f\in V,\quad f_{\big|_{\Omega_0}}\equiv g_{\big|_{\Omega_0}}.
    \end{equation*}
\end{definition}
This notion of coincidence is adapted to the study of singularities (choosing $\Gamma$ being the set where the singularity is supported). 
This definition can be interpreted intuitively as follows: as soon as we lay at positive distance from the singularity supported on $\Gamma$ then the notion of $V$-weak solution coincides with the standard notion of weak solution (taking $W=\cD(\Omega;\RR^d)$ in the definition above).

To better sustain the relevancy of this notion of $V$-weak solution, but also to give a better intuition, we mention the following easy property:

\begin{property}
    let $u\in\cD'(\Omega;\RR^d)$ be a $V$-weak solution of some equation (see definition~\ref{defi:V weak}), with $V$ being $\Gamma$-coincident with $\cD(\Omega:\RR^d)$ for some $\Gamma\subseteq\Omega$. 
    
    Then $u$ is a solution in the sense of distributions in $\Omega\setminus\Gamma$ :  \begin{equation*}
           u\in \cD'(\Omega\setminus\Gamma;\RR^d)\quad\text{ is such that: }\qquad\forall\;\phi\in\cD(\Omega\setminus\Gamma;\RR^d),\qquad\cF(u,\phi)=0.
    \end{equation*}
\end{property}
\noindent
In other words, we have the following hierarchy :
\begin{equation*}
    \text{solution in }\cD'(\Omega)\quad\Longrightarrow\quad V\text{-weak solution (}\Gamma\text{-coincident})\quad\Longrightarrow\quad \text{solution in }\cD'(\Omega\setminus\Gamma).
\end{equation*}
Note also that $\cD(\Omega\setminus\Gamma)$ is $\Gamma$-coincident with $\cD(\Omega)$. Nevertheless, in practice the set $V$ should be chosen larger so that the least possible information is lost in the neighborhood of the singularity.

\subsubsection{Global Existence of $V$-Weak Solutions in the Critical Case}
With this weaker notion of solution at hand and this commutator formulation, we are able to prove the following theorem:

\begin{theorem}[Global $V$-weak solutions to the \emph{quasi} vortex-wave if $s=1/2$]\label{thrm:VFaible}
    Assume that $s=1/2$. Let $\theta_0\in L^1\cap L^\infty(\RR^2)$, let $z_0\in\RR^2$ and let $T>0$. Let $\chi\in L^1(\RR^2)$ such that $\int\chi=1.$ Then there exists $V\subseteq\cD([0,T)\times\RR^2)$ such that:\vspace{0.1cm}

    $(i)$ There exists $\theta\in L^\infty([0,T);L^1\cap L^\infty(\RR^2))$ and $z\in\cC^{0,1}([0,T);\RR^2)$ a $V$-weak solution to the vortex-wave system in the sense that for all test functions $\psi\in V$:
    \begin{equation}\label{eq:Vweak forumulation}\begin{split}
        \int_0^T\int_{\RR^2}\Bigg(\frac{\partial\psi}{\partial t}(t,x)+v(t,x)&\cdot\Big[(-\Delta)^s,\nabla\psi(t,x)\Big](-\Delta)^{-s}\\&+H(t,x)\cdot\nabla\psi(t,x)\Bigg)\theta(t,x)\;\dd x\,\dd t+\int_{\RR^2}\theta_0(x)\,\psi(0,x)\,\dd x=0,
    \end{split}
    \end{equation}
where $v$ and $H$ are given by~\eqref{eq:StrongSystem}, 
and we have for all $t\in[0,T)$,
\begin{equation}\label{eq:Euler Vortex Wave weak 2}
    z(t)=z_0+\int_0^t\int_{\RR^2}v\big(t',x\big)\chi(x-z(t'))\,\dd x\,\dd t'.
\end{equation}

\vspace{0.1cm}
$(ii)$ The set $V$ is such that the product $H\cdot\nabla\psi$ lays in $L^1(\RR^2).$ Moreover, the set $V$ is dense in $L^p(\R^2)$ for all $p\in[1,+\infty]$, dense in $W^{1,p}(\R^2)$ for all $p<+\infty$. \vspace{0.3cm}

$(iii)$ The set $V$ is $\Gamma$-coincident with $\cD([0,T)\times\RR^2)$ for $\Gamma$ the trajectory of the point-vortex:
\begin{equation*}
    \Gamma\,:=\,\bigcup_{t\in[0,T)}\big\{(t,z(t))\big\}\quad\subseteq\;[0,T)\times\RR^2.
\end{equation*}

$(iv)$ If moreover a solution $\theta$ of~\eqref{eq:Vweak forumulation} has regularity $\cC^1$ on $[0,T)\times\RR^2$ and is constant on a neighborhood of $z(t)$ for all $t\in[0,T)$, then $(\theta,z)$ is a classical solution on $[0,T)$ to the the \emph{quasi} vortex-wave with $s=1/2.$
\end{theorem}
Point $(iv)$ of this theorem is a main and standard property for any weak solution of any equation. 
This property can be seen as one of the main interest of the notion of weak solution because the notion of weak solution makes sense as long as we can reconstruct a classical solution whenever the weak solution turns out to be smooth enough.
What we prove at Point $(iv)$ of this theorem is that this important property is also satisfied by our notion of $V$-weak solutions.
The fact that the weak and $V$-weak solutions share this property is the main argument in favor of this new weaker notion of solution.

Concerning Point $(iii)$ of the theorem, this property of $\Gamma$-coincidence means, from an intuitive point-of-view, that the solution is a distribution ``outside'' the singularity (meaning at positive distance from the point-vortex) and is not defined ``inside'' the singularity.

The proof of Theorem~\ref{thrm:VFaible} is done in Section~\ref{sec:proofVfaible}.

\section{Functional Analysis and Regularization}

\subsection{Fourier Analysis and Besov spaces}
We first define a very convenient tool for Fourier analysis in a context of non-linear equations which is the Littlewood-Paley analysis. For a more detailed presentation, see~\cite{Bahouri_Chemin_Danchin_2011}. Recall that the Fourier transform of $f:\RR^2\to\RR$ is defined by
\begin{equation*}
    \widehat{f}(\xi)=\cF f(\xi):=\int_{\RR^2}f(x)e^{-ix\cdot\xi}\dd x
\end{equation*}

\begin{definition}[Littlewood-Paley decomposition]\label{d:LPDecomposition}
    We consider a nonnegative function $\phi\in\cD(\RR^2)$ that is radial and such that $\phi(\xi)=1$ if $|\xi|\leq1/2$ and $\phi(\xi)=0$ if $|\xi|\geq1.$ We define $\psi(\xi):=\phi(\xi/2)-\phi(\xi)$, and we call this function the \textit{Fourier cut-off function}.

    We then define the $j$-th \textit{homogeneous dyadic block} of the Littlewood-Paley decomposition for a given tempered distribution $f\in\cS'(\RR^2)$ by
    \begin{equation*}
        \dot{\triangle}_jf:=\cF^{-1}\bigg(\psi\Big(\frac{\xi}{2^j}\Big)\,\cF(f)\bigg).
    \end{equation*}

    We also define the \textit{non-homogenous dyadic blocks} $(\triangle_j)_{j \geq -1}$ in the following way: when $j \geq 0$, we set $\triangle_j = \dot{\triangle}_j$, and for $j = -1$, we note
    \begin{equation}\label{eq:LFOperator}
        \triangle_{-1} f = \mc F^{-1} \left( \phi(\xi) \mc F (f) \right).
    \end{equation}
    With the cut-off function $\phi$, we define a \textit{low frequencies cutting operator} by setting, for $j \in \mathbb{Z}$
    \begin{equation}\label{eq:LPSOperator}
        \mathrm{S}_j f := \mc F \big[ \phi (2^{-j} \xi) \mc F (f) \big].
    \end{equation}
    We also define a \textit{high frequencies cutting operator} by:
    \begin{equation*}
        \mathrm{H}_jf:=(1-\mathrm{S}_j)f=\sum_{k=j}^{+\infty}\dot{\triangle}_kf.
    \end{equation*}
\end{definition}

With the Littlewood-Paley decomposition, it is possible to define the Besov spaces:

\begin{definition}[Inhomogeneous Besov spaces]
    For $s\in\RR,$ and $p,q\in[1,+\infty]$, the distribution $f\in\cS'$ belongs to the inhomogeneous Besov space $B^s_{p,q}$ if and only if the following Besov norm is finite:
    \begin{equation*}
        \|f\|_{B^s_{p,q}}\;:=\;\| \triangle_{-1} f \|_{L^p} + \bigg(\sum_{j=0}^{+\infty}2^{jqs}\,\|\triangle_jf\|_{L^p}^q\bigg)^\frac{1}{q}.
    \end{equation*}
    With this norm, the space $B^s_{p,q}$ is a Banach space.
\end{definition}

\begin{definition}[Homogeneous Besov spaces]
    For $s\in\RR,$ and $p,q\in[1,+\infty]$, the distribution $f\in\cS'$ belongs to the homogeneous Besov space $\dot{B}^s_{p,q}$ if and only if the following Besov semi-norm is finite:
    \begin{equation*}
        \|f\|_{\dot{B}^s_{p,q}}\;:=\bigg(\sum_{j=-\infty}^{+\infty}2^{jqs}\,\|\dot{\triangle}_jf\|_{L^p}^q\bigg)^\frac{1}{q}.
    \end{equation*}
\end{definition} 
It is standard to prove that $B^s_{2,2}=H^s$ for $s\in\RR$ and that $\dot{B}^s_{2,2}=\dot{H}^s$ in the case $s<1$. One of the most useful properties of homogeneous spaces is their scaling invariance: if $f \in \dot{B}^s_{p, r}$ and $f_\lambda (x) = f(\lambda x)$, then we have $\| f_\lambda \|_{\dot{B}^s_{p, r}} \approx \lambda^{s - 2/p} \| f \|_{\dot{B}^s_{p, r}}$. 

\medskip

We will also need a Bernstein inequality: if a dyadic block $\dot{\triangle}f$ belongs to $L^p$, then it belongs to all $L^q$ spaces for $q \geq p$, but the price to pay is a power of $2^j$, which is roughly the regularity scale of $\dot{\triangle}_j f$. This is formalized in the following Lemma.

\begin{lemma}[see Lemma 2.1, p. 52, in \cite{Bahouri_Chemin_Danchin_2011}]\label{l:Bernstein}
    Consider $1 \leq p \leq q \leq + \infty$, $j \in \mathbb{Z}$, and $f \in \mc S'$ such that $\dot{\triangle}_j f \in L^p$. Then the following inequality holds:
    \begin{equation*}
        \big\| \dot{\triangle}_j f \big\|_{L^q} \leq C(d) 2^{j d \left( \frac{1}{p} - \frac{1}{q} \right)} \big\| \dot{\triangle}_j f \big\|_{L^p}.
    \end{equation*}
\end{lemma}

We are able to give a meaning to the convolution with the kernel $K_s:x\mapsto x^\perp/|x|^{4-2s}$ for all value of $s$ provided that the function $\theta$ is $H^1$:

\begin{proposition}\label{p:FourierMultiplierKS}
    Consider $s \in (0, 1)$. Then the operator $\nabla^\perp (- \Delta)^{-s}$ defines a bounded Fourier multiplier $H^{1-2s} \longrightarrow L^\infty + L^2$. In other words, the function
    \begin{equation*}
        \what{g}(\xi) = i \xi^\perp |\xi|^{-2s} \what{\theta}(\xi)
    \end{equation*}
    is the Fourier transform a function $g = \triangle_{-1} g + ({\rm Id} - \triangle_{-1}) g$ with
    \begin{equation*}
        \| \triangle_{-1} g \|_{L^\infty} + \big\| ({\rm Id} - \triangle_{-1}) g \big\|_{L^2} \lesssim \| \theta \|_{H^{1-2s}}.
    \end{equation*}
\end{proposition}

\begin{proof}
    We split the function $g$ into two parts, a low frequency one and a high frequency one: we have, by making use of the small frequencies cutting operator of \eqref{eq:LFOperator},
    \begin{equation*}
        \begin{split}
            \what{g}(\xi) & = \mc F \big[ \triangle_{-1} g \big] (\xi) + \mc F \big[ ({\rm Id} - \triangle_{-1}) g \big](\xi) \\
            & = \phi(\xi) i \xi^\perp |\xi|^{-2s} \what{\theta}(\xi) + \big( 1 - \phi(\xi) \big) i \xi^\perp |\xi|^{-2s} \what{\theta}(\xi) := \what{g_1}(\xi) + \what{g_1}(\xi),
        \end{split}
    \end{equation*}
    where $\phi \in \mc D$ is the Littlewood-Paley function from Definition \ref{d:LPDecomposition}. Let us focus on the low frequency part, which will give a $L^\infty$ function: we have
    \begin{equation*}
        \| g_1 \|_{L^\infty} \lesssim \| \what{g_1} \|_{L^1} \lesssim \int \phi(\xi) |\xi|^{1 - 2s} |\what{\theta}(\xi)| { \rm d} \xi.
    \end{equation*}
    Now, the Cauchy-Schwarz inequality provides, as $\phi$ is supported in a ball $B(0, 1)$,
    \begin{equation*}
        \int \phi(\xi) |\xi|^{1 - 2s} |\what{\theta}(\xi)| { \rm d} \xi \lesssim \| \what{\theta} \|_{L^2} \left( \int_0^1 r^{2(1 - 2s)} r {\rm d} r \right)^{1/2} \lesssim \| \theta \|_{L^2}.
    \end{equation*}
    On the other hand, the high frequency term reads
    \begin{equation*}
        \what{g_2}(\xi) = \big( 1 - \phi(\xi) \big) i \xi^\perp |\xi|^{-2s} \what{\theta}(\xi) = \underbrace{\big( 1 - \phi(\xi) \big) \frac{i \xi^\perp |\xi|^{-2s}}{1 + |\xi|^{1 - 2s}}}_{\in L^\infty} \big( 1 + |\xi|^{1 - 2s} \big) \what{\theta}(\xi),
    \end{equation*}
    so that we do indeed have $\| g_2 \|_{L^2} \leq \| \theta \|_{H^{1-2s}}$.
\end{proof}

In the particular case where $s=1/2$, it is possible to give a meaning to the convolution with $K_\frac{1}{2}$ when $\theta$ is only in $L^1$ or $L^2$ (which will be useful for the study of weak solutions):
\begin{lemma}\label{lem:K 1/2}  The following estimate holds: \begin{equation*}
        \|K_\frac{1}{2}\star\theta\|_{L^2}\lesssim\|\theta\|_{L^2}.
    \end{equation*}
\end{lemma}

\begin{proof}
    The Fourier multiplier associated to the convolution with $K_\frac{1}{2}$ is the symbol $\xi^\perp/|\xi|$, which is bounded. The $L^2$ estimate is then obtained using the Fourier-Plancherel formula.
\end{proof}

\subsection{Commutator formulation}\label{sec:commutator}
To start with, we prove the commutator formulation stated by Lemma~\ref{lem:Commutator formulation}:
\begin{proof}
    On the one hand we simply write
    \begin{equation}\label{Sologne}
        I:=\int_{\RR^2}\theta(x)\,K_s\!\star\theta(x)\cdot\nabla\varphi(x)\,\dd x=\int_{\RR^2}K_s\!\star\theta(x)\cdot\nabla\varphi(x)\,(-\Delta)^s(-\Delta)^{-s}\theta(x)\,\dd x
    \end{equation}
    On the other hand, recalling that $K_s\star=\nabla^\perp(-\Delta)^{-s}$, an integration by part gives,
    \begin{equation*}
        I=-\int_{\RR^2}(-\Delta)^{-s}\theta(x)\,\nabla^\perp\cdot\big(\theta\nabla\varphi\big)(x)\,\dd x=-\int_{\RR^2}(-\Delta)^{-s}\theta(x)\,\nabla^\perp\theta(x)\cdot\nabla\varphi(x)\,\dd x,
    \end{equation*}
    where for the second equality we used the Schwarz theorem that states $\nabla^\perp\cdot\nabla\phi=0.$ We now remark that $(-\Delta)^{-s}$ is a convolution operator. Writting $1=(-\Delta)^{-s}(-\Delta)^{s}$, The Fubini theorem then gives:
    \begin{equation}\label{Verdun}
        I=-\int_{\RR^2}\nabla^\perp\theta(x)\cdot(-\Delta)^{-s}(-\Delta)^{s}\Big((-\Delta)^{-s}\theta(x)\,\nabla\varphi\Big)(x)\,\dd x=-\int_{\RR^2}K_s\star\theta(x)\cdot(-\Delta)^{s}\Big((-\Delta)^{-s}\theta\,\nabla\varphi\Big)(x)\,\dd x
    \end{equation}
    Summing~\eqref{Sologne} and~\eqref{Verdun} gives~\eqref{eq:commutator formulation}.
\end{proof}

From this reformulation of the non-linear term as a commutator, we have the following properties that are classical for the analysis of SQG (see given references for proofs).
We make widely use of these results throughout the article:

\begin{lemma}[Lemma 2.2 in~\cite{Marchand_2008}]\label{lem:calderon}
    For $s\geq1/2$, define the singular integral operator $J:=[(-\Delta)^s,\nabla\phi]$. For $1<p<+\infty$, there exists $C_p\geq0$ such that
    \begin{equation*}
        \|J(f)\|_{L^p}\;\leq\;C_p\,\|\nabla^2\phi\|_{L^\infty}\|f\|_{L^p}.
    \end{equation*}
\end{lemma}

\begin{lemma}[See Proposition 2.1 in \cite{Chae_Constantin_Cordoba_Gancedo_Wu_2012}]\label{p:CCCGW_Commutator}
    Consider $s \in \R$, $j \in \{ 1, 2 \}$ and the operator $A_j = \partial_j (- \Delta)^{- s}$. For any $\varepsilon > 0$ the bound
    \begin{equation*}
        \big\| [A_j, g]f \big\|_{L^2} \leq C(s) \| f \|_{L^2} \left\| \what{(- \Delta)^{(1 - 2s)/2} g} \right\|_{L^1} + C(s) \| (- \Delta)^{-s} f \|_{L^2} \left\| \what{(- \Delta)^{1/2} g} \right\|_{L^1} .
    \end{equation*}
    holds as long as all the norms above make sense (and are finite).
\end{lemma}

\subsection{Regularization of the Kernel}\label{sec:reg}

In this paragraph, we define the regularization $K_{s, \varepsilon}$ of the kernel $K_s$ that we will be using throughout the article, as well as the approximate system that will serve to construct solutions. 

\medskip

First of all, we fix a radial smooth cut-off function $\chi \in C^\infty$ such that $0 \leq \chi \leq 1$ and
\begin{equation*}
    \chi(x) = 1 \text{ for } |x| \leq \frac{1}{2}, \qquad \chi(-x) = \chi(x) \qquad \text{and} \qquad {\rm supp}(\chi) \subset B(0, 1).
\end{equation*}
For any radius $\varepsilon > 0$, we also define a scaled version of this cut-off
\begin{equation*}
    \chi_\varepsilon(x) = \chi \left( \frac{x}{\varepsilon} \right).
\end{equation*}
This allows to truncate smoothly the kernel $K_s$, and so get rid of the singularity. We define the truncated kernel $K_{s, \varepsilon}$ by 
\begin{equation*}
    K_{s, \varepsilon}(x) := \big( 1 - \chi_\varepsilon(x) \big) K_s (x),
\end{equation*}
so that $K_{s, \varepsilon}(x) = K_s(x)$ for all $|x| \geq \varepsilon$. We also have the inequality 
\begin{equation}\label{bio}
     |K_{s,\varepsilon}(x)|\leq |K_{s}(x)|=\frac{c_s}{|x|^{3-2s}}.
\end{equation}
In particular, the convergence $K_{s, \varepsilon} \longrightarrow K_s$ holds uniformly locally in $\R^2 \setminus \{ 0 \}$ as the truncation parameter goes to zero $\varepsilon \rightarrow 0^+$ (see also Lemma \ref{l:KernelConvergence}). Also note that $K_{s, \varepsilon}$ is a smooth odd function $K_{s, \varepsilon}(x) = - K_{s, \varepsilon}(-x)$. Finally, we should note that the approximate kernel $K_{s, \varepsilon}$ is divergence-free, which is the content of the following lemma.

\begin{lemma}\label{l:HRisDivFree}
    For any $t \geq 0$ such that $R(t) > 0$, we have $\D (K_{s, \varepsilon}) = 0$.
\end{lemma}

\begin{proof}
    By writing the divergence of $K_{s, \varepsilon}$, we see that
    \begin{equation*}
        \D (K_{s, \varepsilon}) = - \nabla \chi_\varepsilon \cdot K_s + (1 - \chi_\varepsilon) \D (K_s).
    \end{equation*}
    The second term in this equation is zero because $K_s$ is smooth and divergence-free on the support of $1 - \chi_\varepsilon$. The first term is also zero because the vector $K_s(x)$ is colinear to $x^\perp$ while the radial symmetry of the cut-off function $\chi$ implies that $\nabla \chi_\varepsilon(x)$ is colinear to $x$.
\end{proof}

\medskip

With the help of the kernel $K_{s, \varepsilon}$, we may define an approximate system, for which the existence of solutions is easier. More precisely, we consider the following set of PDEs
\begin{equation}\label{eq:ApproxSystem1}
    \begin{cases}
        \partial_t \theta_\varepsilon + \D \big( (v_\varepsilon + \mc H_\varepsilon) \theta_\varepsilon \big) = 0\\
         \der{}{t} z_\varepsilon(t) = v_\varepsilon\big( t, z_\varepsilon(t) \big) \\
        v_\varepsilon = K_{s, \varepsilon} \star \theta_\varepsilon \\
        \mc H_\varepsilon(t,x) = K_{s, \varepsilon} \big( x - z_\varepsilon(t) \big),
    \end{cases}
\end{equation}
which we equip with the initial datum 
\begin{equation}\label{eq:RegularizedThetaZero}
    \theta_{\varepsilon}(0) = \theta_{0, \varepsilon} := \chi_{1/\varepsilon} . (S_{j_\varepsilon} \theta_0),
\end{equation}
where $S_j$ is the Littlewood-Paley approximation operator \eqref{eq:LPSOperator} and $j_\varepsilon \geq 1/\varepsilon$. In this way, the initial datum is smooth and compactly supported.

\medskip




Since the velocity field $v_\varepsilon$ is divergence-free ($\nabla \cdot \nabla^\perp=0$ by the Stokes theorem), the Liouville theorem~\cite{Arnold_1989} states that for all measurable set $X\subseteq\RR^2:$
\begin{equation*}
    \cL^2(X)=\cL^2(\Phi^t_\varepsilon X),
\end{equation*}
where $\cL^2$ is the Lebesgue measure on $\RR^2$ and $t\geq0\mapsto
\Phi^t_\varepsilon$ is the flow map associated to the velocity field $v_\varepsilon+H_\varepsilon$. This velocity field is divergence-free and the active scalar is transported by the flow. We deduce the following conservation property for all $p\in[1,+\infty]$:
\begin{equation}\label{conservation law epsilon}
    \forall\,t>0,\qquad\|\theta_\varepsilon(t,.)\|_{L^p}=\|\theta_0\|_{L^p}.
\end{equation}

Before examining the existence of classical solutions, we first show that the approximate system \eqref{eq:ApproxSystem1} has global (approximate) solutions.

\begin{proposition}\label{prop:global smooth approx sol}
    Consider any $\varepsilon > 0$ and the initial data $(\theta_{0, \varepsilon}, z_0)$ defined by \eqref{eq:RegularizedThetaZero}. Then the approximate system has a global smooth solution $(\theta_\varepsilon, z_\varepsilon)$.
\end{proposition}

\begin{proof}
    Since the approximate system \eqref{eq:ApproxSystem1} is fully non-linear PDE system, it is not immediately obvious that it has global solutions. To construct these, we resort to a second approximation: define for all $N \geq 1$ the operator $E_N$ by its Fourier transform
    \begin{equation*}
        \forall f \in L^2, \qquad \what{E_N f}(\xi) := \mathds{1}_{|\xi| \leq N} \what{f}(\xi)
    \end{equation*}
    and consider the system of equations
    \begin{equation}\label{eq:AppCauchyLipschitz}
        \begin{cases}
            \partial_t \theta_N + E_N \D \big( (v_N + \mc H_N) \theta_N \big) = 0\\
            \der{}{t} z_N(t) = v_N(t, z_N(t))\\
            v_N = K_{s, \varepsilon} \star \theta_N\\
            \mc H_N (t, x) = K_{s, \varepsilon}(x - z_N(t)).
        \end{cases}
    \end{equation}
    We associate to \eqref{eq:AppCauchyLipschitz} the initial data
    \begin{equation*}
        \big( \theta_N(0), z_0 \big) := \big( E_N \theta_{0, \varepsilon}, z_0)
    \end{equation*}
    with the perspective of using the Cauchy-Lipschitz theorem to construct a solution $(\theta_N, z_N)$. With that in mind, we define the Banach space
    \begin{equation*}
        X_N := \Big\{ (\theta, z) \in L^2 \times \R^2, \quad {\rm supp}(\what{f}) \subset B (0, N) \Big\},
    \end{equation*}
    which we endow with the norm
    \begin{equation*}
        \big\| (\theta, z) \big\|_{X_N} := \| \theta \|_{L^2} + |z|.
    \end{equation*}

    \medskip

    In order to apply the Cauchy-Lipschitz theorem, we must check that the map
    \begin{equation*}
        F : \left(
        \begin{array}{c}
             \theta  \\
             z 
        \end{array}
        \right) \longmapsto \left( 
        \begin{array}{c}
             E_N \D \big( \big( \theta \star K_{s, \varepsilon} + K_{s, \varepsilon} (x-z) \big) . \theta \big)  \\
              K_{s, \varepsilon} \star \theta (t, z)
        \end{array}
        \right)
    \end{equation*}
    is locally Lipschitz on $X_N \longrightarrow X_N$. We therefore consider two elements $(\theta_1, z_1)$ and $(\theta_2, z_2)$ of $X_N$ and study the difference $\big\| F(\theta_2, z_2) - F(\theta_1, z_1) \big\|_{X_N}$. We have
    \begin{equation*}
        \begin{split}
            \big\| F(\theta_2, z_2) - F(\theta_1, z_1) \big\|_{X_N} & \lesssim N \big\| (\theta_2 - \theta_1) \star K_{s, \varepsilon} . \theta_2 \big\|_{L^2} + N \big\| \theta_1 \star K_{s, \varepsilon} . (\theta_2 - \theta_1) \big\|_{L^2} \\
            & + N \big\| \big( K_{s, \varepsilon} (x - z_2) - K_{s, \varepsilon}(x-z_1) \big).\theta_2 \big\|_{L^2} + N \big\| K_{s, \varepsilon} (x - z_1) . (\theta_2 - \theta_1) \big\|_{L^2}\\
            & + \big| K_{s, \varepsilon} \star \theta_2 (t, z_2) - K_{s, \varepsilon} \star \theta_2 (t, z_1) \big| + \big| K_{s, \varepsilon} \star \theta_2 (t, z_1) - K_{s, \varepsilon} \star \theta_1 (t, z_1) \big| \\
            & := \sum_{k = 1}^6 M_k.
        \end{split}
    \end{equation*}
    We start by noting that the approximate kernel $K_{s, \varepsilon}$ is a $L^2$ function for any value of $\varepsilon > 0$. Therefore, we bound the first two terms $M_1$ and $M_2$ by
    \begin{equation*}
        \begin{split}
            M_1 + M_2 & \lesssim N \big\| (\theta_2 - \theta_1) \star K_{s, \varepsilon} \big\|_{L^\infty} \| \theta_2 \|_{L^2} + N \| \theta_1 \star K_{s, \varepsilon} \|_{L^\infty} \| \theta_2 - \theta_1 \|_{L^2} \\
            & \lesssim 2 N \| K_{s, \varepsilon} \|_{L^2} \| \theta_2 \|_{L^2} \| \theta_2 - \theta_1 \|_{L^2}.
        \end{split}
    \end{equation*}
    For the next two terms, we note that the kernel $K_{s, \varepsilon}$ also is a bounded Lipschitz function. This implies the following bounds for $M_3$ and $M_4$,
    \begin{equation*}
        M_3 + M_4 \lesssim N \| \nabla K_{s, \varepsilon} \|_{L^\infty} |z_2 - z_1| \| \theta_2 \|_{L^2} + N \| K_{s, \varepsilon} \|_{L^\infty} \| \theta_2 - \theta_1 \|_{L^2}.
    \end{equation*}
    Finally, for the last two terms, $M_5$ and $M_6$, we also use the fact that $K_{s, \varepsilon} \in W^{1, \infty}$ to write
    \begin{equation*}
        \begin{split}
            M_5 + M_6 & \lesssim \| \nabla K_{s, \varepsilon} \star \theta_2 \|_{L^\infty} |z_2 - z_1| + \| K_{s, \varepsilon} \star (\theta_2 - \theta_1) \|_{L^\infty} \\
            & \lesssim \| \nabla K_{s, \varepsilon} \|_{L^2} \| \theta_2 \|_{L^2} |z_2 - z_1| + \| K_{s, \varepsilon} \|_{L^2} \| \theta_2 - \theta_1 \|_{L^2}.
        \end{split}
    \end{equation*}
    Putting together these three estimates, we see that
    \begin{equation*}
        \big\| F(\theta_2, z_2) - F(\theta_1, z_1) \big\|_{X_N} \lesssim C(s, \varepsilon, N) \| \theta_2 \|_{L^2} \big( \| \theta_2 - \theta_1 \|_{L^2} + |z_2 - z_1| \big),
    \end{equation*}
    so that the map is indeed $X_N \longrightarrow X_N$ Lipschitz. Accordingly, the Cauchy-Lipschitz theorem provides the existence of a (unique) maximal solution
    \begin{equation*}
        (\theta_N, z_N) : (T_N^-, T_N^+) \longrightarrow X_N.
    \end{equation*}

    \medskip
    
    We now prove that this solution is global, or, since we only are interested in non-negative times, that $T_N^+ = + \infty$. For this, we remark that the conservation of $L^2$ norms also applies to solutions of \eqref{eq:AppCauchyLipschitz}, since the approximate kernel $K_{s, \varepsilon}$ is divergence-free, according to Lemma \ref{l:HRisDivFree}. Consequently, the vector field $v_N + \mc H_N$ is divergence-free and
    \begin{equation*}
        \frac{1}{2} \der{}{t} \int |\theta_N(t)|^2 \dd x = 0,
    \end{equation*}
    and this implies that there is conservation of  $L^2$ norms $\| \theta_N (t) \|_{L^2} = \| \theta_N(0) \|_{L^2} \leq \| \theta_0 \|_{L^2}$. Furthermore, the approximate velocity $v_N$ also satisfies the bound
    \begin{equation*}
        \| v_N \|_{W^{1, \infty}} \lesssim \| K_{s, \varepsilon} \|_{H^1} \| \theta_N \|_{L^2}.
    \end{equation*}
    Using this in the differential equation satisfied by the particle trajectory, we see that
    \begin{equation*}
        \left| \der{}{t} z_N \right| \lesssim \| K_{s, \varepsilon} \|_{H^1} \| \theta_N \|_{L^2} |z_N| \leq C(s, \varepsilon) \| \theta_0 \|_{L^2} |z_N|.
    \end{equation*}
    Grönwall's inequality then gives the global estimate
    \begin{equation}\label{eq:UniformBoundeEpsilon}
        \big| z_N(t) \big| \leq \big( \| \theta_0 \|_{L^2} + |z_0| \big) e^{C(s, \varepsilon) t \| \theta_0 \|_{L^2}}.
    \end{equation}
    Note that the upper bound in this inequality is independent on $N$. In particular, we deduce that the approximate solutions $(\theta_N, z_N)$ remain bounded in the space $X_N$, so that they cannot blow-up in finite time and $T_N^+ = + \infty$.

    \medskip

    Now, in order to construct the solution $(\theta_\varepsilon, z_\varepsilon)$, we prepare to take the limit $N \rightarrow + \infty$. From estimate \eqref{eq:UniformBoundeEpsilon}, we already know that the sequence $(\theta_N, z_N)_N$ is uniformly bounded in $L^\infty(L^2) \times L^\infty_{\rm loc}$ with respect to the approximation parameter $N \geq 1$. Consequently, since the kernel $K_{s, \varepsilon}$ lies in $H^\ell$ for any $\ell \in \R$, we see that, for any fixed time $T > 0$,
    \begin{equation}\label{ordre et beaute}
        \begin{split}
            \| \partial_t \theta_N \|_{L^\infty_T (H^{-1})} & \leq \| v_N + \mc H_N \|_{L^\infty} \| \theta_N \|_{L^2} \\
            & \leq \big( \|K_{s, \varepsilon} \|_{L^\infty} + \|K_{s, \varepsilon} \|_{L^2} \| \theta_N \|_{L^2} \big) \| \theta_N \|_{L^2} \\
            & \leq C(s, \varepsilon) \big( \| \theta_0 \|_{L^2}^2 + \| \theta_0 \|_{L^2} \big).
        \end{split}
    \end{equation}
    The sequence $(\theta_N)$ is then bounded in the space $W^{1, \infty}_{\rm loc}(H^{-1})$. We deduce the strong convergence of the sequence (up to an extraction),
    \begin{equation*}
        \theta_N \tend{N}{+ \infty} \theta_\varepsilon \qquad \text{in } C^0_T (H^{- 1 - \delta}_{\rm loc})
    \end{equation*}
    for any $\delta > 0$. In particular, the regularity of the approximate kernel $K_{s, \varepsilon} \in H^\ell$ yields the convergence
    \begin{equation}\label{eq:AAvelocityConvergence}
        v_N \tend{N}{+ \infty} v_\varepsilon \qquad \text{in } C^0_T (H^{\ell'}_{\rm loc})
    \end{equation}
    for any $\ell' \in \R$, from which we deduce convergence of the product
    \begin{equation*}
        \theta_N v_N \tend{N}{+ \infty} \theta_\varepsilon v_\varepsilon \qquad \text{in } C^0_T(H^{-1 - \delta}_{\rm loc}).
    \end{equation*}
    In the same way, by using the differential equation $\der{}{t}z_N = v_N(t, z_N)$, we see that the approximate particle trajectories are bounded in the space $W^{1, \infty}(\R_+ ; \R^2)$, and Ascoli's theorem then shows that the sequence $(z_N)$ is uniformly convergent (up to an extraction). Associating this with the convergence \eqref{eq:AAvelocityConvergence} provides, for any $t \in [0, T),$
    \begin{equation*}
        \begin{split}
            \big| v_N(t, z_N) - v_\varepsilon (t, z_\varepsilon) \big| & \leq \big| v_N (t, z_N) - v_N (t, z_\varepsilon) \big| + \big| v_N (t, z_\varepsilon) + v_\varepsilon (t, z_\varepsilon) \big| \\
            & \leq \| z_N - z_\varepsilon \|_\infty \| v_N \|_{H^3} + \| v_N - v_\varepsilon \|_{H^2} \\
            & \tend{N}{+ \infty} 0 \; .
        \end{split}
    \end{equation*}
    The inequality above proves that the functions $v_N (t, z_N)$ converge uniformly to $v_\varepsilon (t, z_\varepsilon)$. In exactly the same way, we see that the functions $\mc H_N (t, x) = K_{s, \varepsilon}(x - z_N(t))$ converge to $\mc H_{\varepsilon} (t, x) = K_{s, \varepsilon} (x - z_\varepsilon(t))$. 

    \medskip
    
    Overall, we see that the functions $(\theta_\varepsilon, v_\varepsilon, \mc H_\varepsilon)$ solve the approximate system \eqref{eq:ApproxSystem1} with initial datum $\theta_{0, \varepsilon}$. Note that because $v_\varepsilon = K_{s, \varepsilon} * \theta_\varepsilon$ is a smooth function, the $C^\infty$ regularity of the initial datum is propagated by the flow.
    
\end{proof}


\section{Local Existence of Classical Solutions}\label{sec:existence_strong}

The section is devoted to the proof of existence of solutions as stated in Theorem~\ref{thrm:strong_solutions}-$(i)$. More precisely, we prove the following.
\begin{theorem}\label{t:StrongExistence}
   Consider $k \geq 4$ and a set of vortex-wave initial data $(\theta_0, z_0) \in H^k(\R^2) \times \R^2$ such that $\theta_0$ is constant (equal to, say a fixed $\beta \in \R$) on a ball $B(z_0, R_0)$ centered at the point vortex and of positive radius $R_0 > 0$.
   
   Then there exists a time $T > 0$ such that the vortex-wave system \eqref{eq:StrongSystem} has a solution $(\theta, z) \in L^\infty([0, T); H^k(\R^2)) \times C^1([0, T); \R^2)$. In addition, the time $T$ may be chosen such that
\begin{equation*}
    T \geq \frac{C}{E_0^{6+3k-6s} + E_0^{1/2}},
\end{equation*}
where $E_0 = \| \theta_0 \|_{H^k}^2 + \frac{1}{R_0}$.

\end{theorem}

\subsection{\textsl{A Priori} Estimates}\label{sec:apriori}

This first paragraph is devoted to the establishment of \textsl{a priori} estimates for our system: assuming that the solution is regular, we will derive high regularity estimates on some time interval. These are summarized in the following proposition.

\begin{proposition}\label{p:APriori}
    Consider an initial datum $(\theta_0, z_0)$ and a solution $(\theta, z)$ that are very regular, both in terms of smoothness and integrability.  Furthermore, assume that there is a $R_0 > 0$ such that $\theta_0$ is constant, say equal to a fixed $\beta \in \R$, on the ball $B(z_0, R_0)$. 
    
    For every time, we set $R(t) \geq 0$ to be the largest radius $r \geq 0$ such that the function $\theta(t)$ is constant on the ball $B \big( z(t), R(t) \big)$, see \eqref{def:R(t)}. Then, for every integer $k \geq 4$, there exists a time $T > 0$ such that the bound
    \begin{equation*}
        E(t) := \big\| \theta(t) \big\|_{H^k}^2 + \frac{1}{R(t)} \leq C \left( \| \theta_0 \|_{H^k}^2 + \frac{1}{R_0} \right)
    \end{equation*}
    holds for every $0 \leq t \leq T$. Moreover, the time $T$ satisfies the inequality
    \begin{equation*}
        T \geq \frac{C}{E_0^{6+3k-6s} + E_0^{1/2}}.
    \end{equation*}
\end{proposition}

Adopt the notation of Proposition \ref{p:APriori}. The proof is based on an energy method: consider a multi-index $\alpha \in \NN^2$ whose length is $k := |\alpha| = \alpha_1 + \alpha_2$. By integrating the transport equation against $\partial^{2\alpha} \theta$, we obtain an energy estimate involving the $\dot{H}^k$ norm of $\theta$, namely
\begin{equation}\label{eq:energyEstimate}
\begin{split}
    \frac{1}{2} \frac{\rm d}{ \dt} \int |\partial^\alpha \theta|^2 \dx & = - \int \partial^\alpha \theta . \partial^\alpha \D (\theta . \nabla^\perp (- \Delta)^{-s} \theta) \dx - \int \partial^\alpha \theta . \partial^\alpha \D (\theta H) \\
    & = J_1 + J_2.
\end{split}
\end{equation}

\subsubsection{Estimates for the First Integral}

We will first start by focusing on the integral $J_1$, involving a trilinear product of derivatives of $\theta$, and turn afterwards our attention to $J_2$. This is the purpose of the following lemma.

\begin{lemma}\label{l:energyIneqEQ1}
    With the notation defined above, the inequality
    \begin{equation}\label{eq:energyIneqEQ1}
        \big| J_1(t) \big| \leq C(s, k) \big\| \theta(t) \big\|_{H^k}^3
    \end{equation}
    holds for all times $t \geq 0$.
\end{lemma}

\begin{proof}
We distribute the derivatives in the product $\theta . \nabla^\perp (- \Delta)^{-s} \theta$ by means of the binomial formula:
\begin{equation*}
\int \partial^\alpha \theta . \partial^\alpha \D (\theta . \nabla^\perp (- \Delta)^{-s} \theta) \dx = \sum_{\gamma \leq \alpha} \left(
\begin{array}{c}
     \alpha  \\
     \gamma 
\end{array}
\right) \int \partial^\alpha \theta . \D \big( \partial^\gamma \theta . \partial^{\alpha - \gamma} \nabla^\perp (- \Delta)^{-s} \theta \big) \dx,
\end{equation*}
where the binomial coefficient is defined by
\begin{equation*}
\left(
\begin{array}{c}
     \alpha  \\
     \gamma 
\end{array}
\right) := \frac{\gamma !}{\alpha ! (\alpha - \gamma)!} = \frac{\gamma_1 ! \gamma_2 !}{\alpha_1 ! \alpha_2 ! (\alpha_1 - \gamma_1) ! (\alpha_2 - \gamma_2) !}.
\end{equation*}
We distinguish between four cases in the sum. Firstly, all derivatives could be of order lower than $k$, so that we may easily bound the integral by using only $\| \theta \|_{H^k}$. More precisely, assume that $2 \leq |\gamma| \leq k - 1$, so that
\begin{equation*}
\begin{split}
\left| \int \partial^\alpha \theta . \D \big( \partial^\gamma \theta . \partial^{\alpha - \gamma} \nabla^\perp (- \Delta)^{-s} \theta \big) \dx \right| & \leq \| \partial^\alpha \theta \|_{L^2} \| \partial^{\alpha - \gamma} \nabla^\perp (- \Delta)^{-s} \theta \|_{L^\infty} \| \nabla \partial^\gamma \theta \|_{L^2} \\
& \leq \| \theta \|_{H^k} \| \partial^{\alpha - \gamma} \nabla^\perp (- \Delta)^{-s} \theta \|_{L^\infty} \| \theta \|_{H^{k}}
\end{split}
\end{equation*}
In order to bound the $L^\infty$ norm in the previous inequality, we make use of the Sobolev embedding $H^{1 + \varepsilon} \subset L^\infty$, which holds for any $\varepsilon > 0$. Because the exponent $s > 0$ is positive, the operator $\nabla^\perp (- \Delta)^{-s}$ is of order smaller than $1$,  so that
\begin{equation*}
\| \partial^{\alpha - \gamma} \nabla^\perp (- \Delta)^{-s} \theta \|_{L^\infty} \leq \| \partial^{\alpha - \gamma} \nabla^\perp (- \Delta)^{-s} \theta \|_{H^{1 + \varepsilon}} \leq \| \theta \|_{H^{\sigma}},
\end{equation*}
where the index $\sigma = (k-2) + (1 - s) + (1 + \varepsilon)$ is smaller than $\sigma \leq k$ if $\varepsilon$ is chosen small enough. Therefore we have the bound
\begin{equation}\label{eq:thetaIntegral1}
\left| \int \partial^\alpha \theta . \D \big( \partial^\gamma \theta . \partial^{\alpha - \gamma} \nabla^\perp (- \Delta)^{-s} \theta \big) \dx \right| \leq \| \theta \|^3_{H^k} \leq \| \theta \|_{H^k}^3.
\end{equation}

Secondly, when we have $|\gamma| = 1$, similar computations  lead to
\begin{equation}\label{eq:thetaIntegral2}
\begin{split}
\left| \int \partial^\alpha \theta . \D \big( \partial^\gamma \theta . \partial^{\alpha - \gamma} \nabla^\perp (- \Delta)^{-s} \theta \big) \dx \right| & \leq \| \partial^\alpha \theta \|_{L^2} \| \partial^\gamma \nabla \theta \|_{L^\infty} \| \partial^{\alpha - \gamma} \nabla^{\perp} (- \Delta)^{-s} \theta \|_{L^2} \\
& \leq \| \theta \|_{H^k} \| \partial^\gamma \nabla \theta \|_{H^{1+\varepsilon}} \| \theta \|_{H^{k-s/2}} \leq \| \theta \|_{H^k}^3
\end{split}
\end{equation}
provided that the exponent $k$ be greater than $k \geq 4 > 3 + \varepsilon$.

The third case of interest is when $\gamma = \alpha$, so that the integral cancels because the velocity field $v = \nabla^\perp (- \Delta)^{-s} \theta$ is divergence free:
\begin{equation}\label{eq:thetaIntegral3}
\int \partial^\alpha \theta . \D \big( \partial^\alpha \theta .  \nabla^\perp (- \Delta)^{-s} \theta \big) \dx = 0.
\end{equation}

Finally, we are left with the case where $\gamma = 0$, so the full weight of the derivatives bears on the velocity field $v = \nabla^\perp (- \Delta)^{-s} \theta$. In order to circumvent the loss of derivatives induced by the operator $A = \nabla^\perp (- \Delta)^{-s}$, we take advantage of the commutator structure in the equation: by noting $A = (A_1, A_2) = \big(- \partial_2 (- \Delta)^{-s}, \partial_1 (- \Delta)^{-s} \big)$, we define two skew-symmetric operators $A_1$ and $A_2$ so that
\begin{equation*}
\begin{split}
\int \partial^\alpha \theta . \D \big( \theta . \partial^{\alpha} \nabla^\perp (- \Delta)^{-s} \theta \big) \dx & = \int \partial^\alpha \theta . \partial_j \theta . \partial^\alpha A_j \theta \dx \\
& = - \int A_j (\partial^\alpha \theta . \partial_j \theta) . \partial^\alpha \theta \dx \\
& = \frac{1}{2} \int \partial^\alpha \theta . \big[ \partial_j \theta, A_j \big] \partial^\alpha \theta \dx,
\end{split}
\end{equation*}
where there is an implicit sum on the repeated index $j \in \{ 1, 2\}$. Using Proposition \ref{p:CCCGW_Commutator} with $\sigma = -s$ to bound the commutator appearing in the last integral, we get, for any $\varepsilon > 0$,
\begin{equation*}
\begin{split}
\left| \int \partial^\alpha \theta . \D \big( \theta . \partial^{\alpha} \nabla^\perp (- \Delta)^{-s} \theta \big) \dx \right| & \lesssim \| \partial^\alpha \theta \|_{L^2} \, \big\| \big[ \partial_j \theta, A_j \big] \partial^\alpha \theta \big\|_{L^2} \\
& \lesssim \| \partial^\alpha \theta \|_{L^2}^2 \int |\xi|^{2 - 2s} |\what{\theta}(\xi)| {\rm d} \xi + \| \partial^\alpha \theta \|_{L^2} \| \partial^\alpha \theta \|_{\dot{H}^{-2s}} \int |\xi|^2 |\what{\theta}(\xi)| {\rm d} \xi .
\end{split}
\end{equation*}
In order to bound the integrals above, note that the function $\xi \mapsto |\xi|^{2 - 2s}$ is bounded on a neighborhood of $\xi = 0$. As a consequence, for every $\varepsilon > 0$, we may use the Cauchy-Schwarz inequality to write
\begin{equation*}
    \begin{split}
        \int \big( |\xi|^{2 - 2s} + |\xi|^2 \big) |\what{\theta}(\xi)| {\rm d} \xi & \leq \int \frac{|\xi|^2}{1 + |\xi|^{3 + \varepsilon}} \big( 1 + |\xi| \big)^{3 + \varepsilon} |\what{\theta}(\xi)| {\rm d} \xi \\
        & \leq C(\varepsilon) \| \theta \|_{H^{3 + \varepsilon}}.
    \end{split}
\end{equation*}
Furthermore, the since $k = |\alpha| \geq 4 > 2s$, the homogeneous norm $\| \theta \|_{\dot{H}^{2s}}$ is bounded by $\| \theta \|_{H^k}$. We deduce
\begin{equation}\label{eq:thetaIntegral4}
    \left| \int \partial^\alpha \theta . \D \big( \theta . \partial^{\alpha} \nabla^\perp (- \Delta)^{-s} \theta \big) \dx \right| \lesssim \| \theta \|_{H^k}^3.
\end{equation}
Putting together the bounds \eqref{eq:thetaIntegral1}, \eqref{eq:thetaIntegral2}, \eqref{eq:thetaIntegral3} and \eqref{eq:thetaIntegral4}, we then get the following estimate for the second summand in \eqref{eq:energyEstimate}:
\begin{equation*}
\left| \int \partial^\alpha \theta . \partial^\alpha \D (\theta . \nabla^\perp (- \Delta)^{-s} \theta) \dx \right| \lesssim \| \theta \|_{H^k}^3.
\end{equation*}
\end{proof}

\subsubsection{Radius Estimates}

Now, we take care of the last integral $J_2$ in \eqref{eq:energyEstimate} involving the influence of the point vortex through $H(t,x)$. This will create no particular problem: the point vortex is assumed to be in a region where $\theta$ is constant, and this will allow us to ignore any contribution of $H(t,x)$ near the singular point $z(t)$. In order to implement this argument, for any time, we need to prove that the $\theta$ remains constant near $z(t)$. We prove the following lemma.

\begin{lemma}\label{lem:remain_constant_near_PV}
Let $(\theta,z)$ be a strong solution on $[0,T)$ to \eqref{eq:StrongSystem} (in the sense of Theorem~\ref{thrm:strong_solutions}) and assume that
\begin{equation*}
    \theta_0 \equiv \beta \qquad \text{ on } B\big(z(0),R_0\big),
\end{equation*}
for some $\beta \in \RR$. Recall that $R$, defined in relation~\eqref{def:R(t)}, is given by
\begin{equation*}
    R(t) = \sup \big\{ r \ge 0 \; , \; \theta(t) \equiv \beta \text{ on } B(z(t),r) \big\}.
\end{equation*}
Then for every $t \in [0,T)$,
\begin{equation*}
    R(t) \ge R(0) \exp \left(- \int_0^t \| \nabla v(\tau,\cdot) \|_{L^\infty} \dd \tau\right).
\end{equation*}
In addition, we also have the inequality
\begin{equation}\label{eq:Rdefinition}
\der{}{t} \left( \frac{1}{R(t)} \right) \lesssim \frac{\| \nabla v \|_{L^\infty}}{R(t)} \lesssim \frac{\| \theta \|_{H^3}}{R(t)}.
\end{equation}
\end{lemma}

\begin{proof}
Let $x \in \R^2$ such that $\theta(t,x) \neq \beta$. Let $\tau \mapsto X(\tau)$ be the trajectory of the fluid particle passing in $x$ at time $t$. Then
\begin{equation*}\begin{split}
    \left. \der{}{\tau}|X(\tau) - z(\tau)|\right|_{\tau = t} & = \frac{x-z(t)}{|x-z(t)|} \cdot \big(v(x,t) + H(x,t) - v(z(t),t)\big) \\
    & = \frac{x-z(t)}{|x-z(t)|} \cdot \big(v(x,t) - v(z(t),t)\big) \\
    & \ge - |x-z(t)| \| \nabla v \|_{L^\infty} \\
    & \ge - R(t)\| \nabla v \|_{L^\infty}
\end{split}\end{equation*}
where we used the fact that $\big(x-z(t)\big) \cdot H(x,t)= 0$. Since this inequality holds for every value of $x$ such that $\theta(t,x)\neq\beta,$ we deduce that it holds also for the points $x$ that belong to the border of $\{x\in\RR^2:\theta(t,x)\neq\beta\}$ (by continuity of the flow). In particular, it holds for the points $x$ that are at distance $R(t)$ from the point-vortex $\delta_{z(t)}$. Therefore:
\begin{equation}\label{eq:est_der_R}
    \der{}{t} R(t) \ge -R(t)\| \nabla v \|_{L^\infty}.
\end{equation}
Using the Grönwall inequality,
\begin{equation*}
    R(t) \ge R(0) \exp \left(-\int_0^t \| \nabla v(\tau) \|_{L^\infty} \dd \tau\right).
\end{equation*}
Finally, inequality \eqref{eq:Rdefinition} follows from \eqref{eq:est_der_R} and the Sobolev embeddings $H^k \subset H^3 \subset W^{1, \infty}$.
\end{proof}

\subsubsection{Estimates for the Second Integral}

We are now in a favorable position to bound the integral $J_2$. Our estimates are contained in the lemma immediately below.

\begin{lemma}\label{l:energyIneqEQ2}
    With the notation of Proposition \ref{p:APriori}, the inequality
    \begin{equation}\label{eq:energyIneqEQ2}
        \left| \int \partial^\alpha \theta . \partial^\alpha \D (\theta . H) \dx \right| \lesssim \| \theta \|_{H^k}^2 \left( \frac{1}{R(t)^{4 - 2s}} + \frac{1}{R(t)^{2 + k - 2s}} \right)
    \end{equation}
    holds for all times $t \geq 0$.
\end{lemma}

\begin{proof}
We now define a smooth truncation of the singular contribution of the point vortex $H(t,x)$. Let $\chi \in \mc D$ be the radial and smooth cut-off function defined in Subsection \ref{sec:reg}, and set, for any time such that $R(t) > 0$, 
\begin{equation*}
\widetilde{H}_R(t,x) = \big( 1 - \chi_{R(t)} (x - z(t)) \big) H(t,x), \qquad \text{where } \chi_R(y) := \chi \left( \frac{y}{R} \right).
\end{equation*}
Lemma \ref{l:HRisDivFree} shows that $\widetilde{H}_R$ is a divergence-free function. This and the fact that $\theta$ is constant on the ball $B \big( z(t), R(t) \big)$, and so $\nabla \theta  = 0$ on that area, let us obtain the equalities
\begin{equation}\label{eq:DivFactoring}
\D (\theta H) = H \cdot \nabla \theta = \widetilde{H}_R \cdot \nabla \theta = \D( \theta \widetilde{H}_R).
\end{equation}
We come back to \eqref{eq:DivFactoring}, thanks to which we express the last integral in the energy estimate \eqref{eq:energyEstimate} as
\begin{equation*}
\int \partial^\alpha \theta . \partial^\alpha \D (\theta \widetilde{H}_R) \dx = \sum_{|\gamma| \leq |\alpha|} \left(
\begin{array}{c}
     \alpha \\
     \gamma
\end{array}
\right) \int \partial^\alpha \theta . \D \big( \partial^\gamma \theta . \partial^{\alpha - \gamma} \widetilde{H}_R \big) \dx.
\end{equation*}
In estimating the integrals above, there are three cases. Firstly, when $\gamma = \alpha$, we have
\begin{equation*}
\int \partial^\alpha \theta . \nabla \partial^\alpha \theta \cdot \widetilde{H}_R \dx = 0,
\end{equation*}
since the approximate kernel $K_{s,\varepsilon}$ is divergence free, according to Lemma~\ref{l:HRisDivFree}.
Secondly, we explore the case where $|\gamma| \leq k-3$. We then have, thanks to the Sobolev embedding $H^2 \subset L^\infty$,
\begin{equation*}
\begin{split}
\left| \int \partial^\alpha \theta . \D \big( \partial^\gamma \theta . \partial^{\alpha - \gamma} \widetilde{H}_R \big) \dx \right| & \leq \| \partial^\alpha \theta \|_{L^2} \| \partial^{\alpha - \gamma} \widetilde{H}_R \|_{L^2} \| \nabla \partial^\gamma \theta \|_{L^\infty} \\
& \lesssim \frac{\| \theta \|_{H^k}^2}{R(t)^{2 + k - |\gamma| - 2s}}\\
& \lesssim \| \theta \|_{H^k}^2 \left( \frac{1}{R(t)^{5 - 2s}} + \frac{1}{R(t)^{2 + k - 2s}} \right).
\end{split}
\end{equation*}
Lastly, in the case where $k-2 \leq |\gamma| \leq k-1$, so that $1 \leq |\alpha - \gamma| \leq 2$, we instead evaluate the $L^\infty$ norm of $\partial^\gamma H_R$ to get
\begin{equation*}
\begin{split}
\left| \int \partial^\alpha \theta . \D \big( \partial^\gamma \theta . \partial^{\alpha - \gamma} \widetilde{H}_R \big) \dx \right| & \leq \| \partial^\alpha \theta \|_{L^2} \| \partial^{\alpha - \gamma} \widetilde{H}_R \|_{L^\infty} \| \nabla \partial^\gamma \theta \|_{L^2} \\
& \lesssim \frac{\| \theta \|_{H^k}^2}{R(t)^{3 + k - |\gamma| - 2s}} \\
& \lesssim \| \theta \|_{H^k}^2 \left( \frac{1}{R(t)^{4-2s}} + \frac{1}{R(t)^{5 - 2s}} \right).
\end{split}
\end{equation*}
Note that because $k \geq 4$, the exponent $4 - 2s$ in this last upper bound is smaller than the exponent $2 + k - 2s$ obtained above. Overall, we have the estimate
\begin{equation*}
\left| \int \partial^\alpha \theta . \partial^\alpha \D (\theta . \widetilde{H}_R) \dx \right| \lesssim \| \theta \|_{H^k}^2 \left( \frac{1}{R(t)^{4 - 2s}} + \frac{1}{R(t)^{2 + k - 2s}} \right).
\end{equation*}
\end{proof}

\subsubsection{Proof of the \textsl{A Priori} Estimates}

Finally, with the help of the lemmas above, we can prove Proposition \ref{p:APriori}.

\begin{proof}[Proof (of Proposition \ref{p:APriori})]
Plugging estimates \eqref{eq:energyIneqEQ1} from Lemma \ref{l:energyIneqEQ1} and \eqref{eq:energyIneqEQ2} from Lemma \ref{l:energyIneqEQ2} along with \eqref{eq:Rdefinition} into inequality \eqref{eq:energyEstimate}, we obtain a differential inequality:
\begin{equation*}
\begin{split}
    \der{}{t} \left( \| \theta \|_{H^k}^2 + \frac{1}{R(t)} \right) & \lesssim \| \theta \|_{H^k}^3 + \frac{\| \theta \|_{H^3}}{R(t)} + \| \theta \|_{H^k}^2 \left( \frac{1}{R(t)^{4 - 2s}} + \frac{1}{R(t)^{2 + k - 2s}} \right)\\
    & \lesssim \| \theta \|^3_{H^k} + \frac{1}{R(t)^{3/2}} + \frac{1}{R(t)^{6 + 3k - 6s}} \\
    & \lesssim \left( \| \theta \|_{H^k}^2 + \frac{1}{R(t)} \right)^{3/2} + \left( \| \theta \|_{H^k}^2 + \frac{1}{R(t)} \right)^{6 + 3k - 6s}.
    \end{split}
\end{equation*}
Define $E(t) := \| \theta \|_{H^k}^2 + R(t)^{-1}$ and set, for notational convenience, $q+1 := 6 + 3k - 6s$. Then, by integrating the inequality above, we obtain, for $T \geq 0$
\begin{equation*}
E(T) \leq E(0) + C \int_0^T \Big( E(t)^{3/2} + E(t)^{q+1} \Big) \dt.
\end{equation*}
Let $[0, T^\ast]$ be the time interval on which the integral above stays smaller than the initial data:
\begin{equation*}
T^\ast := \sup \left\{ T > 0, \quad \int_0^T \Big( E(t)^{3/2} + E(t)^{q+1} \Big) \dt \leq E(0) \right\}.
\end{equation*}
Then, as long as $0 \leq T \leq T^\ast$, the energy $E(t)$ satisfies the inequality
\begin{equation}\label{eq:EnergyEstimate}
    E(T) \leq C E(0),
\end{equation}
from which we infer
\begin{equation*}
\int_0^T \Big( E(t)^{3/2} + E(t)^{q_2+1} \Big) \dt \leq CT E(0)^{3/2} + C T E(0)^{q+1},
\end{equation*}
and consequently
\begin{equation*}
T^\ast \geq \frac{C}{E(0)^q + E(0)^{1/2}}.
\end{equation*}
\end{proof}

\subsection{Construction of a solution}

This Subsection is dedicated to the construction of a local solution of \eqref{eq:StrongSystem}. 
The proof relies on the regularized dynamics defined in Subsection \ref{sec:reg}, whose solutions will naturally satisfy the \textsl{a priori} estimates of Proposition \ref{p:APriori}.

\subsubsection{Step 1: Approximate Solutions and Uniform Bounds}

\begin{lemma}\label{l:ApproxSolutionsExist}
    Assume that the initial data $(\theta_0, z_0)$ are as in Proposition \ref{t:StrongExistence}. Then for all $\varepsilon > 0$, the approximate system \eqref{eq:ApproxSystem1} equipped with the initial data $(\theta_{0, \varepsilon}, z_0)$ has a global solution $(\theta_\varepsilon, z_\varepsilon)$. In addition, this solution satisfies the bounds of Proposition \ref{p:APriori}, namely if $R_\varepsilon(t)$ is such that 
\begin{equation*}
\theta_\varepsilon(t, x) = \beta \qquad \text{for } |x - z_\varepsilon(t)| \leq R_\varepsilon (t),
\end{equation*}
then we have a time $T_\varepsilon > 0$ such that the inequalities
\begin{equation}\label{eq:UniformBounds}
\| \theta_\varepsilon(t) \|_{H^k} + \frac{1}{R_\varepsilon(t)} \lesssim \| \theta_0 \|_{H^k} + \frac{1}{R_0}
\end{equation}
hold for $t \in [0, T_\varepsilon[$. In addition $T_\varepsilon$ satisfies the lower bound $T_\varepsilon \geq C^{-1} T^\ast$ independently of the approximation parameter $\varepsilon > 0$. 
\end{lemma}

\begin{proof}
    The only non-trivial point is the existence of a global approximate solution, which is given by Proposition \ref{prop:global smooth approx sol}. The exact same computations leading to Proposition \ref{p:APriori} show that they also follow the \textsl{a priori} estimates.
\end{proof}

\subsubsection{Step 2: Convergence of the Regularized Kernels}

Before showing that the approximate solutions $(\theta_\varepsilon, z_\varepsilon)$ converge when $\varepsilon \rightarrow 0^+$, we must study the regularized kernel $K_{s, \varepsilon}$. We already know that $K_{s, \varepsilon} \longrightarrow K_s$ locally uniformly in $\R^2 \setminus \{ 0 \}$, however, we will need a more precise norm convergence. This is the purpose of the following lemma.
\begin{lemma}\label{l:KernelConvergence}
    The kernels $K_{s, \varepsilon}$ are uniformly bounded in the homogeneous Besov space $\dot{B}^{2s - 1}_{1, \infty}$. In addition, for any $\sigma < 2s - 1$, we have the convergence
    \begin{equation}\label{eq:KernelConvergence}
        K_{s, \varepsilon} \longrightarrow K_s \qquad \text{in } \dot{B}^{\sigma}_{1, \infty}.
    \end{equation}
\end{lemma}

\begin{proof}
    As seen from equation \eqref{def:K_s}, the kernel $K_s$ is a homogeneous function of degree $2s - 3$. A simple scaling argument (see Proposition 2.21, p. 65 in \cite{Bahouri_Chemin_Danchin_2011}) shows that $K_s \in \dot{B}^{2s - 1}_{1, \infty}$. Now, from \eqref{eq:KepsRegularization}, we see that the difference $K_s - K_{s, \varepsilon}$ reads
    \begin{equation*}
        \begin{split}
            K_s (x) - K_{s, \varepsilon} (x) & = \chi \left( \frac{x}{\varepsilon} \right) \frac{x^\perp}{|x|^{4 - 2s}} \\
            & = \left( \frac{1}{\varepsilon} \right)^{3 - 2s} \chi \left( \frac{x}{\varepsilon} \right) \frac{(x/\varepsilon)^\perp}{|x / \varepsilon|^{4 - 2s}} \\
            & = \left( \frac{1}{\varepsilon} \right)^{3 - 2s} (\chi K_s) \left( \frac{x}{\varepsilon} \right).
        \end{split}
    \end{equation*}
    This shows that the difference $K_s - K_{s, \varepsilon}$ is a scaled version of the function $\chi K_s$. Therefore, for $\sigma \in \R$, we have, by the scaling properties of homogeneous Besov spaces,
    \begin{equation*}
        \big\| K_s - K_{s, \varepsilon} \big\|_{\dot{B}^\sigma_{1, \infty}} \lesssim \left( \frac{1}{\varepsilon} \right)^{\sigma - (2s - 1)} \big\| \chi K_s \big\|_{\dot{B}^\sigma_{1, \infty}}.
\end{equation*}
    In particular, the convergence \eqref{eq:KernelConvergence} holds as soon as $\sigma < 2s - 1$, provided that we indeed have $\chi K_s \in \dot{B}^\sigma_{1, \infty}$. This follows from the fact that homogeneous and non-homogeneous Besov spaces coincide for functions supported in a given compact subset of $\R^2$ (see Proposition 2.93, p. 108, in \cite{Bahouri_Chemin_Danchin_2011}).
\end{proof}

\subsubsection{Step 3: Time Compactness and Convergence of the Temperatures}

We start by showing strong convergence of the approximate (potential) temperatures $(\theta_\varepsilon)$, as an application of Ascoli's theorem.

\begin{lemma}\label{l:ConvTheta}
    Let $(\theta_\varepsilon, z_\varepsilon)$ be the approximate solution defined in Lemma \ref{l:ApproxSolutionsExist}. Then, for any $\ell' < k$, we have the strong convergence (up to an extraction)
    \begin{equation}\label{eq:ThetaConvergence}
        \theta_\varepsilon \longrightarrow \theta \qquad \text{in } L^\infty ([0, T) ; H^{\ell'}_{\rm loc}).
\end{equation}
\end{lemma}

\begin{proof}
While the uniform bounds of Lemma \ref{l:ApproxSolutionsExist} provide space compactness on the sequence $(\theta_\varepsilon)_{\varepsilon > 0}$, we will also need compactness with respect to the time variable: by using the evolution equation for $\theta_\varepsilon$, we see that
\begin{equation*}
    \partial_t \theta = - v_\varepsilon \cdot \nabla \theta_\varepsilon - \chi_{R_\varepsilon} H_\varepsilon \cdot \nabla \theta_\varepsilon.
\end{equation*}
Because $\theta_\varepsilon$ is constant (equal to $\beta$) on a neighborhood of $z_\varepsilon(t)$, the products in the right-hand side are well-defined and we have, for all times $t \in [0, T)$,
\begin{equation}\label{eq:FirstTimeDerivative}
    \begin{split}
        \| \partial_t \theta_\varepsilon \|_{H^{k-1}} & \lesssim \| v_\varepsilon \|_{H^{k-1}} \| \nabla \theta_\varepsilon \|_{H^{k-1}} + \| \chi_{R_\varepsilon} H_\varepsilon \|_{H^{k-1}} \| \nabla \theta_\varepsilon \|_{H^{k-1}} \\
        & \lesssim \| \theta_\varepsilon \|_{H^k}^2 + \frac{1}{R_\varepsilon(t)^{k+1 - 2s}  } \| \theta_\varepsilon \|_{H^k} \\
        & \lesssim \| \theta_0 \|_{H^k}^2 + \frac{1}{R_0^{k+1 - 2s}} \| \theta_0 \|_{H^k}.
    \end{split}
\end{equation}
In particular, we see that the sequence $(\theta_\varepsilon)_{\varepsilon > 0}$ is uniformly bounded in the space $W^{1, \infty}([0, T) ; H^{k-1})$. While this already provides both time and space compactness, it will be convenient to use an interpolation argument to get convergence in a stronger space topology. Consider $\ell \in (k-1, k)$ and $\sigma \in (0, 1)$ such that $\ell = (k-1) \sigma + k (1 - \sigma)$. Then, for $t_1, t_2 \in [0, T)$, we have
\begin{equation*}
    \begin{split}
        \big\| \theta(t_1) - \theta (t_2) \big\|_{H^\ell} & \lesssim \big\| \theta(t_1) - \theta (t_2) \big\|_{H^{k-1}}^\sigma \big\| \theta(t_1) - \theta (t_2) \big\|_{H^k}^{1 - \sigma} \\
        & \lesssim 2^{1 - \sigma} |t_2 - t_1|^\sigma \| \theta \|_{W^{1, \infty} (H^{k-1})}^\sigma \| \theta \|_{L^\infty (H^k)}^{1 - \sigma}.
    \end{split}
\end{equation*}
Consequently, the sequence $(\theta_\varepsilon)_{\varepsilon > 0}$ is bounded in the space $C^{0, \sigma}([0, T) ; H^\ell)$, and we deduce the existence of a $\theta$ with a strong convergence, up to taking a subsequence $(\varepsilon_n)_n$,
\begin{equation*}
    \theta_\varepsilon \longrightarrow \theta \qquad \text{in } L^\infty ([0, T) ; H^{\ell'}_{\rm loc}),
\end{equation*}
for any $\ell' < \ell < k$. Note that in reality, because $\theta_\varepsilon \in C^0([0, T) ; H^{\ell'}_{\rm loc})$, the convergence above is in that space, so that the limit $\theta$ is continuous with respect to the time variable.
\end{proof}

\subsubsection{Step 4: Convergence of the Velocity Fields}

In the same manner, we show the convergence of the approximate velocity fields $(v_\varepsilon)$. Here, there is a slight twist because $v_\varepsilon = K_{s, \varepsilon} \star \theta_\varepsilon$ is a convolution by a kernel depending on the approximation parameter $\varepsilon > 0$. This means that we must also use the convergence properties of the sequence $(K_{s, \varepsilon})$ from Lemma \ref{l:KernelConvergence}.

\begin{lemma}\label{l:ConvVelocity}
    Let $(v_\varepsilon)$ the approximate velocities defined in Lemma \ref{l:ApproxSolutionsExist}. Then we have pointwise convergence (up to an extraction)
    \begin{equation*}
        v_\varepsilon \longrightarrow v \qquad a.e.
    \end{equation*}
    Furthermore, the limit $v$ is given by $v = K_s \star \theta$.
\end{lemma}

\begin{proof}
 Recall from Subsection \ref{sec:reg} that the approximate kernel $K_{s, \varepsilon}$ is given by
\begin{equation}\label{eq:KepsRegularization}
    K_{s, \varepsilon} (x) = \big( 1 - \chi(\varepsilon^{-1} x) \big) K(x).
\end{equation}
and that Lemma \ref{l:KernelConvergence} shows that it converges to $K_\varepsilon$ in the space $\dot{B}^\sigma_{1, \infty}$ for any $\sigma < 2s-1$. We write the difference $v_\varepsilon - v$ as a sum of two terms
\begin{equation*}
    v_\varepsilon - v =  (K_{s, \varepsilon} - K_s) \star \theta_\varepsilon + K_s \star (\theta_\varepsilon - \theta).
\end{equation*}
To handle the first of these terms, we use the Littlewood-Paley decomposition: remembering that $\dot{\triangle}_j \dot{\triangle}_m = 0$ as soon as $|j - m| \geq  2$, we have
\begin{equation*}
    \begin{split}
        \big\| (K_{s, \varepsilon} - K_s) \star \theta_\varepsilon \big\|_{L^\infty} & \leq \sum_{j \in \mathbb{Z}}  \big\| \dot{\triangle}_j (K_{s, \varepsilon} - K_s) \star \theta_\varepsilon \big\|_{L^\infty} \\
        & \lesssim \sum_{|j - m| \leq 1} \big\| \dot{\triangle}_j (K_s - K_{s, \varepsilon}) \star  \dot{\triangle}_m \theta_\varepsilon \big\|_{L^\infty}\\
        & \lesssim \sum_{|j - m| \leq 1}  \big\| \dot{\triangle}_j (K_s - K_{s, \varepsilon}) \big\|_{L^1} \| \dot{\triangle}_m \theta_\varepsilon \|_{L^\infty}
    \end{split}
\end{equation*}
Let us focus on the series with negative indices $j < 0$. The Bernstein-type inequality of Lemma \ref{l:Bernstein} applied to $\dot{\triangle}_m \theta_\varepsilon$ and $\dot{\triangle}_m (\theta_\varepsilon - \theta)$ provides, for any $\sigma < 2s - 1$, 
\begin{equation*}
    \begin{split}
        \sum_{j < 0, \; |j - m| \leq 1} & \big\| \dot{\triangle}_j (K_s - K_{s, \varepsilon}) \big\|_{L^1} \| \dot{\triangle}_m \theta_\varepsilon \|_{L^\infty} \\
        & \lesssim \sum_{j < 0 \; |j - m| \leq 1} 2^{-j \sigma} \| K_s - K_{s, \varepsilon} \|_{\dot{B}^\sigma_{1, \infty}} 2^{j} \| \dot{\triangle}_m \theta_\varepsilon \|_{L^2} \\
        & \lesssim \| K_s - K_{s, \varepsilon} \|_{\dot{B}^\sigma_{1, \infty}} \| \theta_\varepsilon \|_{L^2} \sum_{j < 0} 2^{j(1 - \sigma)}.
    \end{split}
\end{equation*}
By remembering that $\sigma < 2s - 1 < 1$, we see that the series above is convergent. Now, we have to deal with nonnegative indices $j \geq 0$. The Bernstein type inequality of Lemma \ref{l:Bernstein} once again yields
\begin{equation*}
    \begin{split}
        \sum_{j \geq 0, \; |j - m| \leq 1} & \big\| \dot{\triangle}_j (K_s - K_{s, \varepsilon}) \big\|_{L^1} \| \dot{\triangle}_m \theta_\varepsilon \|_{L^\infty} \\
        & \lesssim \sum_{j \geq 0, \; |j - m| \leq 1} 2^{-j \sigma} \| K_s - K_{s, \varepsilon} \|_{\dot{B}^\sigma_{1, \infty}} 2^j \| \dot{\triangle}_m \theta_\varepsilon \|_{L^2} \\
        & \lesssim \sum_{j \geq 0, \; |j - m| \leq 1} 2^{-j (\sigma+1)} \| K_s - K_{s, \varepsilon} \|_{\dot{B}^\sigma_{1, \infty}} \|  \theta_\varepsilon \|_{H^2} \\
        & \lesssim \| K_s - K_{s, \varepsilon} \|_{\dot{B}^\sigma_{1, \infty}} \| \theta_\varepsilon \|_{H^2} \sum_{j \geq 0} 2^{-j (\sigma + 1)} \; .
    \end{split}
\end{equation*}
Because $\sigma < 2s - 1 < 1$, the sum above converges. As a consequence, we have shown that
\begin{equation*}
    \big\| (K_{s, \varepsilon} - K_s) \star \theta_\varepsilon \big\|_{L^\infty} \lesssim \| K_{s, \varepsilon} - K_s \|_{\dot{B}^\sigma_{1, \infty}} \| \theta_\varepsilon \|_{H^2} \tend{\varepsilon}{0^+} 0 \, .
\end{equation*}
We now focus on the second term $K_s \star (\theta_\varepsilon - \theta)$. The issue is that we do not have norm convergence of the sequence $(\theta_\varepsilon)$ in the space $H^2$, as Lemma \ref{l:ConvTheta} only provides $H^2_{\rm loc}$ convergence. This means that we cannot show $L^\infty$ convergence of the sequence $(K_s \star \theta_\varepsilon)$. Instead, we will prove weak convergence. We first note that by repeating the computations above, we see that, for any $f \in \mc S$,
\begin{equation}\label{eq:VelocityConvergence}
    \| K_s \star f \|_{L^\infty} \lesssim \| K_s \|_{\dot{B}^\sigma_{1, \infty}} \| f \|_{H^2}.
\end{equation}
This means that, for any $x \in \R^2$, the function $y \mapsto K_s(x - y)$ defines an element of the space $H^{-2}$. Furthermore, the weak-$(\star)$ convergence (up to an extraction)
\begin{equation*}
    \theta_\varepsilon \wtend^* \theta \qquad \text{in } L^\infty_T(H^2),
\end{equation*}
which follows from the fact that $k \geq 4$, implies that
\begin{equation*}
    \begin{split}
        K_s \star \theta_\varepsilon (x) & = \langle K_s (x - \cdot), \theta_\varepsilon \rangle_{H^{-2} \times H^2} \\
        & \tend{\varepsilon}{0^+} \langle K_s (x - \cdot), \theta \rangle_{H^{-2} \times H^2} = K_s \star \theta (x).
    \end{split}
\end{equation*}
Overall, we have proved that, \textbf{up to an extraction,}
\begin{equation*}
    v_\varepsilon = K_{s, \varepsilon} \star \theta_\varepsilon \tend{\varepsilon}{0^+} v = - \nabla^\perp (- \Delta)^{-s} \theta \qquad \text{a.e.}
\end{equation*}
\end{proof}

\subsubsection{Step 5: Convergence of the Vortex Trajectories}

We wish to secure pointwise convergence of the vortex trajectories ODE $z_\varepsilon = v(t, z_\varepsilon)$.

\begin{lemma}
    Consider the sequence of functions $(z_\varepsilon)$ defined in Lemma \ref{l:ApproxSolutionsExist} and the limit velocity from Lemma \ref{l:ConvVelocity}. Then we have the convergence (up to an extraction)
    \begin{equation*}
        z_{\varepsilon} \longrightarrow z(t) \qquad \text{in } C^{1, \alpha}([0, T))
    \end{equation*}
    and the limit $z$ is a solution of the ODE $z' = v(t, z)$.
\end{lemma}

\begin{proof}
Recall from \eqref{eq:FirstTimeDerivative} that the sequence $(\theta_\varepsilon)_{\varepsilon > 0}$ is bounded in $W^{1, \infty}([0, T) ; H^3)$, so that the sequence $(v_\varepsilon)_{\varepsilon > 0}$ is bounded in the spaces
\begin{equation*}
    W^{1, \infty} ([0, T) ; H^2) \subset W^{1, \infty} ([0, T) ; C^{0, \alpha}) \subset C^{0, \alpha} ([0, T) \times \R^2)
\end{equation*}
for any Hölder exponent $0 < \alpha < 1$. The trajectory derivatives $z'_\varepsilon (t) = v_\varepsilon \big( t, z_\varepsilon(t) \big)$ must therefore also be bounded in $C^{0, \alpha} ([0, T))$. This implies the convergence of the sequence $(z_\varepsilon)_{\varepsilon > 0}$ to 
\begin{equation}\label{eq:PVConvergence}
    z_{\varepsilon} \longrightarrow z \qquad \text{in } C^{1, \alpha}([0, T)),
\end{equation}
for some $z \in C^{1, \alpha}([0, T))$. We deduce
\begin{equation*}
    \begin{split}
        \big| v_\varepsilon (t, z(t_\varepsilon)) - v(t, z(t)) \big| & \leq \big| v_\varepsilon (t, z(t_\varepsilon)) - v_\varepsilon(t, z(t)) \big| + \big| v_\varepsilon(t, z(t_\varepsilon)) - v(t, z(t)) \big| \\
        & \leq \| v_\varepsilon \|_{W^{1, \infty}} \big| v(t, z(t_\varepsilon)) - v(t, z(t)) \big| + \big| v_\varepsilon \big( t, z(t) \big) - v \big( t, z(t) \big) \big| \\
        & \tend{\varepsilon}{0^+} 0 \; .
    \end{split}
\end{equation*}
The convergence above follows from Lemma \ref{l:ConvVelocity}. We deduce that the particle trajectory solves the limit ODE, namely
\begin{equation}\label{eq:LimitODE}
    z'(t) = v\big(t, z(t) \big).
\end{equation}
\end{proof}

\subsubsection{Step 6: End of the Proof of Existence}

In this paragraph, we finish the proof of Theorem \ref{t:StrongExistence}. It is mostly a matter of assembling the convergence results from the lemmas above.

\begin{proof}[Proof of Theorem \ref{t:StrongExistence}]
    We begin by studying the convergence of the non-linear terms $v_\varepsilon \cdot \nabla \theta_\varepsilon$ and $\mc H_\varepsilon \cdot \nabla \theta_\varepsilon$ appearing in the first equation of \eqref{eq:ApproxSystem1}. Thanks to Lemmas \ref{l:ConvVelocity} and \ref{l:ConvTheta}, we write
    \begin{equation}\label{eq:ThetaVelocityConvergence}
        v_\varepsilon \cdot \nabla \theta_\varepsilon \longrightarrow v \cdot \nabla \theta \qquad \text{a.e.}.
    \end{equation}
    On the other hand, the product $\mc H_\varepsilon \cdot \nabla \theta_\varepsilon$ requires more attention because the function
    \begin{equation*}
        \mc H_\varepsilon (t, x) = K_{s, \varepsilon} \big( x - z_\varepsilon(t) \big)
    \end{equation*}
    possesses a double dependency on $\varepsilon$, in the kernel $K_{s, \varepsilon}$, and in the argument $x - z_\varepsilon(t)$. However, as the function $\theta_\varepsilon$ is constant around the singularity of $\mc H_\varepsilon$, we only need to care about the values of $\mc H_\varepsilon(t, x)$ for $|x - z_\varepsilon(t)| > R_\varepsilon(t)$. Let us define the target function
    \begin{equation*}
        H(t, x) := K_s \big( x - z(t) \big).
    \end{equation*}
    On the time interval $[0, T)$, both functions $R_\varepsilon(t)$ and $R(t)$ are bounded from below by
    a constant $\eta > 0$ which is independent of $\varepsilon$ (see Lemma \ref{l:ApproxSolutionsExist}). Moreover, from the uniform convergence \eqref{eq:PVConvergence}, 
    we get that the approximate and target point vortices are closer than $\eta$ when $\varepsilon$ is small enough. 
    Thus,
    \begin{equation*}
        |z_\varepsilon(t) - z(t)| \leq \frac{1}{4} \min \big\{ R_\varepsilon(t), R(t) \big\} \qquad \text{for } t \in [0, T).
    \end{equation*}
    Thanks to the uniform bound \eqref{eq:UniformBounds} in Lemma \ref{l:ApproxSolutionsExist}, we deduce that for such $\varepsilon$, both derivatives $\nabla \theta_\varepsilon$ and $\nabla \theta$ vanish on a ball $B(z(t), R_\ast)$ for some constant radius $R_\ast > 0$ depending only on the initial data. Outside of this ball, we get convergence
    \begin{equation*}
        \big\| K_{s, \varepsilon}(x - z_\varepsilon(t)) - K_s(x - z(t)) \big\|_{L^\infty({}^c B(z(t), R_\ast))} \longrightarrow 0 \qquad \text{uniformly on } [0, T),
    \end{equation*}
    from which we infer, together with \eqref{eq:ThetaConvergence},
    \begin{equation}\label{eq:HConvergence}
        \mc H_\varepsilon \cdot \nabla \theta_\varepsilon \longrightarrow H \cdot \nabla \theta \qquad \text{in } L^\infty([0, T) ; L^2).
    \end{equation}

\medskip

Let us summarize what we have obtained so far: putting together all convergences \eqref{eq:ThetaConvergence}, \eqref{eq:ThetaVelocityConvergence}, \eqref{eq:HConvergence} and the fact that $z(t)$ solves the limit ODE \eqref{eq:LimitODE}, we see that $(z, \theta)$ is indeed a solution of the Vortex-Wave system \eqref{eq:StrongSystem}. It only remains to show that $(z, \theta)$ is a solution associated to the initial datum $(z_0, \theta_0)$. The fact that $z(0) = z_0$ is a consequence of the uniform convergence \eqref{eq:PVConvergence} and the equality $z_\varepsilon(0) = z_0$. Concerning the initial value for the scalar unknown, it is a consequence of the time continuity of $\theta$ (see the comments immediately following \eqref{eq:ThetaConvergence}) and the strong convergence of the initial data 
\begin{equation*}
    \theta_{0, \varepsilon} \longrightarrow \theta_0 \qquad \text{in } L^2.
\end{equation*}
These remarks end the proof of Theorem \ref{t:StrongExistence}.
\end{proof}

\subsection{Time Regularity}

In this paragraph, we check that the solution constructed in Theorem \ref{t:StrongExistence} is indeed a classical solution of the PDE problem, \textsl{i.e.} that all derivatives are well-defined in the sense of $C^1$ functions (with respect to time and space). Together with the existence of strong solutions, this completes the proof of the existence part of Theorem \ref{thrm:strong_solutions}.

\begin{proposition}
    Let $k \geq 4$ be an integer and $\theta_0 \in H^k$ an initial datum. Then the solution $(\theta, z) \in L^\infty([0, T) ; H^k) \times C^{k+s-1-\varepsilon}([0, T) ; \R^2)$ constructed in Theorem \ref{t:StrongExistence} is a classical solution (it is $C^1$ in both time and space). More precisely, we have the following regularity: for any $\ell < k-1$, 
    \begin{equation*}
        \theta \in C^1 \big( [0, T) ; H^\ell \big).
    \end{equation*}
\end{proposition}

\begin{proof}
    Be repeating the computations in \eqref{eq:FirstTimeDerivative}, we already know that the solution satisfies $\theta \in W^{1, \infty}([0, T) ; H^{k-1})$ and $z \in C^1([0, T) ; \R^2)$, but to show that the function $\theta$ has $C^1$ regularity in both time and space, we study the second derivative $\partial_t \theta$, which reads
    \begin{equation}\label{eq:SecondTimeDerivative}
        \partial_t^2 \theta = - \partial_t v \cdot \nabla \theta - v \cdot \nabla \partial_t \theta - \partial_t H \cdot \nabla \theta - H \cdot \nabla \partial_t \theta.
    \end{equation}
    
    On the one hand, from $\theta \in W^{1, \infty}([0, T) ; H^{k-1})$, we deduce that $\partial_t v \in L^\infty([0, T) ; H^{k-2})$, so that
    \begin{equation}\label{eq:SndDerEstEQ1}
        \| \partial_t v \cdot \nabla \theta \|_{H^{k-2}} \lesssim \| \partial_t v \|_{H^{k-2}} \| \nabla \theta \|_{H^{k-2}} < + \infty.
    \end{equation}
    Similarly, we have
    \begin{equation}\label{eq:SndDerEstEQ2}
        \| v \cdot \nabla \partial_t \theta \|_{H^{k-2}} \lesssim \| v \|_{H^{k-2}} \| \partial_t \theta \|_{H^{k-1}} < + \infty.
    \end{equation}
    
    On the other hand, we look at the last two terms in the righthandside of \eqref{eq:SecondTimeDerivative}, which involve $H(t,x)$. Because the function $\theta$ is constant, equal to $\beta$, on the ball $B\big( z(t), R(t) \big)$, the derivatives $\nabla \theta$ and $\partial_t \nabla \theta$ must vanish on that ball. Therefore, we have
    \begin{equation*}
        \partial_t H \cdot \nabla \theta + H \cdot \nabla \partial_t \theta = (1 -  \chi_R) \partial_t H \cdot \nabla \theta + (1 -  \chi_R) H \cdot \nabla \partial_t \theta,
    \end{equation*}
    which is a considerable help, as the function $(1 -  \chi_R) H$ is $C^\infty$ spacewise smooth at any given time. Let us compute the derivative $\partial_t H$. We have
    \begin{equation*}
        \partial_t H (t, x) = - \nabla K_s \big( x - z(t) \big) \cdot v \big( t, z(t) \big),
    \end{equation*}
    so that
    \begin{equation}\label{eq:SndDerEstEQ3}
        \begin{split}
            \| (1 -  \chi_R) \partial_t H \cdot \nabla \theta\|_{H^{k-1}} \lesssim \| \nabla \theta \|_{H^{k-1}} \| (1 - \chi_R) \partial_t H \|_{H^{k-1}} & \leq \big\| ( 1 - \chi_{R(t)}(x)) \nabla K_s (x - z(t)) \cdot v (t, z(t)) \big\|_{H^{k-1}} \| \theta \|_{H^{k}} \\
            & \lesssim \frac{1}{R(t)^{k+2 - 2s}} \| \theta \|_{H^{k}}^2 < + \infty.
        \end{split}
    \end{equation}
    In the same way, we also have
    \begin{equation}\label{eq:SndDerEstEQ4}
        \begin{split}
            \big\| \chi_R H \cdot \nabla \partial_t \theta \big\|_{H^{k-2}} & \lesssim \| \chi_R H \|_{H^{k-2}} \| \partial_t \theta \|_{H^{k-1}} \\
            & \lesssim \frac{1}{R(t)^{k+1 - 2s}} \| \partial_t \theta \|_{H^{k-2}} < + \infty.
        \end{split}
    \end{equation}
    
    By putting together \eqref{eq:SndDerEstEQ1}, \eqref{eq:SndDerEstEQ2}, \eqref{eq:SndDerEstEQ3} and \eqref{eq:SndDerEstEQ4}, we find that 
    \begin{equation*}
        \partial_t^2 \theta \in L^\infty \big( [0, T) ; H^{k-2} \big).
    \end{equation*}
    This shows that $\theta$ has $C^1$ regularity with respect to time in the $H^{k-2}$ topology. However, we are not quite finished yet: if we want to show that $\theta$ is $C^1$ with respect to both time and space, then the case $k=4$ requires a bit more attention, as $H^{2} \not \subset C^1$. This issue is resolved by interpolating between the spaces $W^{1, \infty}([0, T) ; H^{k-1})$ and $W^{2, \infty}([0, T) ; H^{k-2})$. For any $\sigma \in [0, 1]$, we set
    \begin{equation*}
        \ell = (k-2) \sigma + (k-1)(1 - \sigma).
    \end{equation*}
    Then, for any times $t_1, t_2 \geq 0$, we may write, 
    \begin{equation*}
        \begin{split}
            \big\| \partial_t \theta (t_2) - \partial_t\theta (t_1) \big\|_{H^\ell} & \leq \big\| \partial_t \theta (t_2) - \partial_t \theta(t_1) \big\|^\sigma_{H^{k-2}} \big\| \partial_t \theta (t_2) - \partial_t \theta (t_1) \big\|_{H^{k-1}}^{1-\sigma} \\
            & \leq 2^{1 - \sigma} |t_2 - t_1|^\sigma \| \partial_t^2 \theta \|^\sigma_{W^{1, \infty}(H^{k-2})} \| \partial_t \theta \|_{L^\infty(H^{k-1})}^{1 - \sigma}
        \end{split}
    \end{equation*}
    Consequently, we deduce Hölder regularity on the time derivative
    \begin{equation*}
        \partial_t \theta \in C^{0, \sigma} \big( [0, T) ; H^\ell \big),
    \end{equation*}
    which provides the space-time $C^1$ regularity we were looking for: when $0 < \sigma < 1$, the space above is a subspace of $C^0([0, T) ; C^1 \cap W^{1, \infty})$.
\end{proof}

\section{Uniqueness and stability of classical solutions}\label{sec:Uniqueness}

In this section, we conclude the proof of Theorem~\ref{thrm:strong_solutions}-$(i)$ and Theorem~\ref{thrm:strong_solutions}-$(ii)$ by proving the following proposition.
\begin{proposition}\label{prop:stabilité}
    Let $T > 0$, $k \ge 4$ and $(\theta_1,z_1)$ and $(\theta_2,z_2)$ be two solutions in $L^\infty([0,T],H^k)$ of~\eqref{eq:StrongSystem}, as constructed in Section~\ref{sec:existence_strong}, such that there exists $\beta_1$, $\beta_2$, $R_1>0$ and $R_2>0$ such that
    \begin{equation*}
        \theta_i \equiv \beta_i \quad \text{ on } B(z_i(0),R_i).
    \end{equation*}
    There exists $\rho > 0$ depending only on $\|\theta_1(0,\cdot)\|_{H^k}$, $\|\theta_2(0,\cdot)\|_{H^k}$, $R_1$, $R_2$ and $T$ such that if
    \begin{equation}\label{hyp:z2-z1}
    |z_2(0) - z_1(0) | < \rho,
    \end{equation}
    then, there exists a time $T_0 \in [0,T)$ such that for every $t \in [0,T_0]$ and for every $\ell \in \NN$ such that $2 \le \ell \le k-2$,
    \begin{equation*}
        \| \theta_2(t,\cdot) - \theta_1(t,\cdot) \|_{H^{\ell}} + |z_2(t)-z_1(t)| \lesssim \| \theta_2(0,\cdot) - \theta_1(0,\cdot) \|_{H^{\ell}} + |z_2(0)-z_1(0)|,
    \end{equation*}
    where the omitted multiplicative constant depends only on $\|\theta_1(0,\cdot)\|_{H^k}$, $\|\theta_2(0,\cdot)\|_{H^k}$, $R_1$, $R_2$, $\ell$ and $T$.
\end{proposition}
Theorem~\ref{thrm:strong_solutions}-$(ii)$ is a simple reformulation of Proposition~\ref{prop:stabilité} and uniqueness of classical solutions of \eqref{eq:StrongSystem} as stated in Theorem~\ref{thrm:strong_solutions}-$(i)$ comes from applying Proposition~\ref{prop:stabilité} to two solutions with same initial datum 
$(\theta_0,z_0)$ with $\theta_0$ constant in a neighborhood of $z_0$ as we will detail in Section~\ref{sec:conclusion_uniqueness}.

\medskip

Moreover in this section we obtain en route the $L^2$-stability when $s \ge 1/2$ as follows.
\begin{proposition}\label{prop:stabilité_L2}
    If in addition to the hypotheses of Proposition~\ref{prop:stabilité} we assume that $s \ge 1/2$, then
    \begin{equation*}
        \| \theta_2(t,\cdot) - \theta_1(t,\cdot) \|_{L^2} + |z_2(t)-z_1(t)| \lesssim \| \theta_2(0,\cdot) - \theta_1(0,\cdot) \|_{L^2} + |z_2(0)-z_1(0)|.
    \end{equation*}
\end{proposition}

\medskip

The plan of the proof of Proposition~\ref{prop:stabilité} is the following. We start by looking for a convenient time $T_0$ such that both point-vortices remain in the constant part of each $\theta_i$. Then we estimate the time derivative of $|z_1(t)-z_2(t)|$ and $\| \theta_2(t,\cdot) - \theta_1(t,\cdot) \|_{H^1}$, then conclude by a Gronwall's type argument.

In the computations, we denote by $\delta \theta = \theta_2 - \theta_1$, $\delta v = v_2 - v_1$ and $\delta z = z_2 - z_1$ the difference functions.

\subsection{Existence of $\rho$ and $T_0$}
We start by proving the following lemma concerning the existence of constants $\rho$ and $T_0$ in Proposition~\ref{prop:stabilité}.
\begin{lemma}\label{lem:choice_rho_T}
    Let $(\theta_1,z_1)$ and $(\theta_2,z_2)$ chosen as in Proposition~\ref{prop:stabilité}. Then there exists a choice of $\rho$ and $T_0$ such that if \eqref{hyp:z2-z1} is satisfied, then for every $t \in [0,T_0]$, for every $i,j \in \{1,2\}$, $\theta_i$ remains constant in the disk $B(z_j(t),\rho)$.
\end{lemma}
\begin{proof}
For both $i \in \{1,2\}$, recall that $\| \nabla v_i(t,\cdot) \|_{L^\infty} \le C \| \theta_i (t,\cdot)\|_{H^3}$, and from relation~\eqref{eq:EnergyEstimate}, that $\| \theta_i (t,\cdot)\|_{H^3}$ is bounded in time on $[0,T]$ by a constant depending only on $\|\theta_i(0)\|_{H^k}$, $R_i$ and $T$. Therefore, there exists a constant $\rho = \rho(\|\theta_1(0)\|_{H^k},\|\theta_2(0)\|_{H^k},R_1, R_2,T)$ such that for every $i \in \{1,2\}$ and for every $t \in [0,T]$,
\begin{equation*}
    R_i\exp \left(- \int_0^t \| \nabla v_i(\tau,\cdot) \|_{L^\infty} \dd \tau\right) > 4\rho.
\end{equation*}
By Lemma~\ref{lem:remain_constant_near_PV} applied successively to $\theta_1$ and $\theta_2$, we infer that for every $t \in [0,T]$,
    \begin{equation}\label{Hamilton}
    \theta_i(t,\cdot) \equiv \beta_i \quad \text{on } B(z_i(t),4\rho).
\end{equation}
There remains to prove that $\theta_1 \equiv \beta_1$ on $B(z_2(t),\rho)$.

We recall that the definition of the regularized kernel $K_{s,\varepsilon}$ and its properties from Subsection \ref{sec:reg}.
We now establish that for every $\eps \in (0,4\rho]$, we have that
\begin{equation}\label{DeVries}
    (K_{s,\rho}\star\theta_i)(z_i) = K_{s,\eps}\star\theta_i(z_i).
\end{equation}
Indeed, by relation~\eqref{Hamilton},
\begin{equation*}\begin{split}
    K_{s,\eps}\star\theta_i(z_i) & = \int_{\R^2} K_{s,\eps}(z_i-y)\theta_i(y,t)\dd y \\
    & = \int_{\R^2\setminus B(z_i,4\rho)} K_{s,\eps}(z_i-y)\theta_i(y,t)\dd y + \beta_i\int_{B(z_i,4\rho)} K_{s,\eps}(z_i-y)\dd y.
\end{split}\end{equation*}
By skew-symmetry, we have that
\begin{equation}\label{eq:local_imply_0}
    \int_{B(0,4\rho)} K_{s,\eps}(y)\dd y = \int_{B(0,4\rho)} K_{s,\rho}(y)\dd y = 0.
\end{equation}
Then, we observe that for every $\eps \in (0,4\rho)$, $K_{s,\eps} \equiv K_{s,\rho}$ on $\R^2 \setminus B(z_i,4\rho)$, which proves \eqref{DeVries}.

Using the same arguments than the ones used to obtain~\eqref{eq:VelocityConvergence}, we have that (see also Lemma~\ref{l:KernelConvergence}), we have that
\begin{equation*}
    \| K_{s, \varepsilon} \star \theta_i - v \|_{L^\infty} \lesssim \| K_{s, \varepsilon} - K_s \|_{\dot{B}^{\sigma}_{1, \infty}} \| \theta_i \|_{H^2} \longrightarrow 0 \qquad \text{as } \varepsilon \rightarrow 0^+,
\end{equation*}
with $\sigma<2s-1$. Since the functions $K_{s, \varepsilon} \star \theta$ are continuous, the convergence $K_{s, \varepsilon} \star \theta_i \longrightarrow v$ holds pointwise. Consequently, we have:
\begin{equation}\label{Tsunoda}
    \big(K_{s,\eps}\star\theta_i(t,\cdot)\big) (z_i(t)) = v_i(t,z_i(t)).
\end{equation}
Gathering relations \eqref{DeVries} and \eqref{Tsunoda}, we obtain that
\begin{equation}\label{Rosberg}
    \der{}{t}z_i(t) = v_i(t,z_i(t)) = (K_{s,\rho}\star\theta_i)(z_i).
\end{equation}

Observing that for every $t \in [0,T_0]$,
\begin{equation*}
    \big\|(K_{s,\rho} \star \theta_i)(t,\cdot) \big\|_{L^\infty} \le \|K_{s,\rho}\|_{L^\infty} \| \theta_i(t,\cdot)\|_{L^1} = \|K_{s,\rho}\|_{L^\infty} \| \theta_i(0,\cdot)\|_{L^1},
\end{equation*}
by taking
\begin{equation*}
    T_0 < \rho \left(\|K_{s,\rho}\|_{L^\infty} \max_{i \in \{1,2\}} \| \theta_i(0,\cdot)\|_{L^1}\right)^{-1},
\end{equation*}
we deduce from relation~\eqref{Rosberg} that for each $i\in \{1,2\}$ and every $t\in [0,T_0]$ that
\begin{equation*}
    |z_i(t)-z_i(0)| \le  \rho.
\end{equation*}
Adding then the assumption~\eqref{hyp:z2-z1} that $|z_1(0)-z_2(0)| < \rho$, we obtain for every $t \in [0,T_0]$ that
\begin{equation}\label{Vettel}
    |z_2(t)-z_1(t)| \le |z_2(t)-z_2(0)| + |z_2(0)-z_1(0)| + |z_1(t) - z_1(0)| < 3\rho.
\end{equation}
Gathering relation \eqref{Hamilton} and \eqref{Vettel}, we obtain that for every $t \in [0,T_0]$, for every $i,j \in \{1,2\}$,
\begin{equation*}
    \theta_i(t,\cdot) \equiv \beta_i \quad \text{on } B\big(z_j(t),\rho\big).
\end{equation*}

\end{proof}

\subsection{Estimate of the distance between the point-vortices}

We then obtain the following estimate on the evolution of the distance between the two point-vortices.
\begin{lemma}\label{lem:der_z_2-z_1} 
Let $(\theta_1,z_1)$ and $(\theta_2,z_2)$ chosen as in Proposition~\ref{prop:stabilité}, let $\rho$ and $T_0$ be given by Lemma~\ref{lem:choice_rho_T} and assume that \eqref{hyp:z2-z1} holds true. Then, for every $t \in [0,T_0]$, we have that
    \begin{equation*}
    \der{}{t}|z_1(t)-z_2(t)| \lesssim |z_2(t) - z_1(t)| + \|\theta_2-\theta_1\|_{L^2}.
\end{equation*}
\end{lemma}
We recall that the omitted constant may only depends on $\|\theta_1(0)\|_{H^k},\|\theta_2(0)\|_{H^k},R_1, R_2,T$.
\begin{proof}
We compute
\begin{equation*}\begin{split}
     \der{}{t}|z_1(t)-z_2(t)| & \le |v_1(t,z_1(t)) - v_2(t,z_2(t)) | \\
     & \le |v_1(t,z_1(t))-v_1(t,z_2(t))| + |\delta v(t,z_2(t))|.
\end{split}\end{equation*}
From one hand we have that
\begin{equation*}
    |v_1(t,z_1(t))-v_1(t,z_2(t))|\le \|\nabla v_1 \|_{L^\infty} |\delta z| \lesssim \|\theta_1\|_{H^3}|\delta z| \lesssim |\delta z|,
\end{equation*}
by recalling that $\|\theta_1\|_{H^3}$ is bounded on $[0,T]$ by hypothesis. 
From the other hand, by Lemma~\ref{lem:choice_rho_T}, $\delta \theta$ is constant on $B(z_2(t),\rho)$ for every $t \in [0,T_0]$ so we have that 
\begin{equation*}\begin{split}
    \delta v(t,z_2(t))= (K_{s} \star \delta\theta(t,\cdot))(z_2(t)) = (K_{s,\rho} \star \delta\theta(t,\cdot))(z_2(t)).
\end{split}\end{equation*}
By the Cauchy Schwartz inequality, we obtain
\begin{equation*}
    |\delta v(t,z_2(t))| \le \| K_{s,\rho} \|_{L^2} \| \delta\theta \|_ {L^2} \lesssim \| \delta \theta \|_{L^2},
\end{equation*}
which concludes the proof of Lemma~\ref{lem:der_z_2-z_1}.
\end{proof}

\subsection{Estimate on the $H^{\ell}$ norm of $\theta_2 - \theta_1$}

We now turn to the estimate of $\delta\theta$.
\begin{lemma}\label{lem:der_theta_2-theta_1} With the same notations as in Lemma~\ref{lem:der_z_2-z_1}, for every $t\in[0,T_0]$ and every $\ell \in \NN$ such that $2\le \ell \le k-2$, we have that
    \begin{equation*}
    \der{}{t} \|\theta_2-\theta_1\|_{H^{\ell}} \lesssim \|\theta_2-\theta_1\|_{H^{\ell}} + |z_2-z_1|.
\end{equation*}
\end{lemma}

\begin{proof}
 Let 
 \begin{equation*}
     H_{i,\rho}(t,x) := K_{s,\rho}\big(x-z_i(t)\big)
 \end{equation*}
 and 
 \begin{equation*}
     u_{i} := v_i + H_{i,\rho}
 \end{equation*}
 for $i \in \{1,2\}$. We drop the dependence in $\rho$ in the notation $u_i$ since it's not relevant for our computations. One should simply remember that the $u_i$ are smooth flows on $[0,T_0]$.
 
By Lemma~\ref{lem:choice_rho_T}, we have that $\nabla \theta_i \equiv 0$ on $B(z_i(t),\rho)$, so that on the whole plane the equality
\begin{equation*}
    H_i\cdot \nabla \theta_i =H_{i,\rho} \cdot \nabla \theta_i
\end{equation*}
is satisfied. The active scalar $\theta_i$ is not only transported by $v_i+H_i$ but also by $u_{i}$. Consequently, by substracting the two transport equations on $\theta_1$ and $\theta_2$ respectively by the flows $u_1$ and $u_2$, we have that
\begin{equation}\label{eq:sum_transport}
    \partial_t \delta \theta + u_1 \cdot \nabla\delta\theta + \delta u \cdot \nabla \theta_2 = 0,
\end{equation}
with the notation $\delta u = u_{2} - u_{1}$. Let us also denote $\delta H_\rho = H_{2,\rho}-H_{1,\rho}$.
We multiply relation~\eqref{eq:sum_transport} by $\delta \theta$ and integrate to obtain that
\begin{equation}\label{Prost}
    \frac{1}{2} \der{}{t} \|\delta\theta\|_{L^2}^2 + \int_{\R^2} \delta v \cdot \nabla \theta_2\, \delta \theta + \int_{\R^2} \delta H_\rho \cdot \nabla \theta_2\, \delta\theta = 0.
\end{equation}
We start with the last term of this equation. Since $\delta H_{\rho}$ is smooth, then
\begin{equation*}
    \left|\int \delta H_\rho \cdot \nabla \theta_2\, \delta\theta\right| \le \|\delta H_{\rho}\|_{L^\infty} \| \nabla \theta_2 \|_{L^2} \|\delta\theta\|_{L^2}.
\end{equation*}
More precisely, from Lemma~\ref{lem:choice_rho_T}, for every $t\in[0,T_0]$ we know that $H_{i,\rho} \cdot \nabla \theta_2 = 0$ in $B(z_i(t),\rho)$ for both $i\in \{1,2\}$, so that
\begin{equation*}\begin{split}
    \| \delta H_\rho \|_{L^\infty} = \sup_{x\in \R^2} | K_{s,\rho}(x-z_2(t)) - K_{s,\rho}(x-z_1(t))|  \le \| \nabla K_{s,\rho} \|_{L^\infty} |z_2(t)-z_1(t)|.
\end{split}\end{equation*}
Recalling then that $\|\theta_2\|_{H^4}$ is bounded on $[0,T_0]$ by a constant depending only on $\|\theta_2(0)\|_{H^4}$, $R_2$ and $T$, we obtain that
\begin{equation*}
    \left|\int \delta H_\rho \cdot \nabla \theta_2\, \delta\theta\right| \lesssim |\delta z| \| \delta\theta\|_{L^2}.
\end{equation*}
In order to continue our estimates, in the middle term in relation~\eqref{Prost}, one has to control a norm of $\delta v$ in terms of $\delta \theta$. In general, norms of $\delta v$ cannot be controlled by $\|\delta \theta\|_{L^2}$ only so we use Proposition~\ref{p:FourierMultiplierKS} to get that
\begin{equation*}\begin{split}
    \left|\int \delta v \cdot \nabla \theta_2 \,\delta \theta\right| & \le \left| \int \triangle_{-1} \delta v \cdot \nabla \theta_2 \delta \theta\right| +  \left|\int (I-\triangle_{-1}) \delta v \cdot \nabla \theta_2 \delta \theta \right| \\
    & \le \|\triangle_{-1} \delta v \|_{L^\infty} \|\nabla \theta_2\|_{L^2} \| \delta \theta \|_{L^2} + \|(I-\triangle_{-1} \delta v) \|_{L^2} \|\nabla \theta_2\|_{L^\infty} \| \delta \theta \|_{L^2} \\
    & \lesssim \| \theta_2 \|_{H^3}  \| \delta \theta\|_{H^{1-2s}} \| \delta \theta\|_{L^2} \\
    & \lesssim \| \delta \theta\|_{H^{1-2s}} \| \delta \theta\|_{L^2}.
\end{split}\end{equation*}
In the case $s \ge 1/2$, we obtain that
\begin{equation}\label{eq:pour_stabilité_L2}
    \der{}{t} \|\delta \theta \|_{L^2} \lesssim |\delta z| + \| \delta \theta\|_{L^2},
\end{equation}
from which we will obtain Proposition~\ref{prop:stabilité_L2}.
We now continue our computations to prove $H^{\ell}$ stability in the general case $0 < s < 1$. Yet we proved in general that
\begin{equation}\label{eq:Prost_v2}
    \der{}{t} \|\delta \theta \|_{L^2}^2 \lesssim \big(|\delta z| + \| \delta \theta\|_{H^{1-2s}}\big)  \|\delta\theta\|_{L^2}.
\end{equation}
Let $\ell \in \NN$ such that $2s \le \ell \le k-2$. Similarly to what we did in Section~\ref{sec:apriori}, for any $\alpha \in \NN^2$ such that $|\alpha| = \ell$, we have from applying $\partial^\alpha$ to relation \eqref{eq:sum_transport}, multiplying by $\partial^\alpha \delta \theta$ and integrating that
\begin{equation}\label{eq:Senna}
\begin{split}
    \frac{1}{2} \der{}{t} \| \partial^\alpha \delta \theta \|_{L^2}^2 & = - \int \partial^\alpha (u_1 \cdot \nabla \delta \theta) \partial^\alpha \delta \theta = \int \partial^\alpha ( \delta u \cdot \nabla \theta_2) \partial^\alpha \delta \theta \\
    & := - \sum_{\gamma \le \alpha} \binom{\alpha}{\gamma} (I_\gamma + J_\gamma + L_\gamma)
\end{split}
\end{equation}
where, using the Leibniz rule, for every $(0,0) \le \gamma \le \alpha$,
\begin{equation*} \begin{split}
    &I_\gamma = \int \partial^\gamma u_1  \cdot  \partial^{\alpha-\gamma}\nabla\delta\theta \partial^\alpha \delta \theta, \\
    &J_\gamma = \int \partial^\gamma \delta v \cdot \partial^{\alpha - \gamma} \nabla \theta_2 \partial^\alpha \delta \theta, \\
    & L_\gamma = \int \partial^\gamma \delta H_{\rho} \cdot \partial^{\alpha - \gamma} \nabla \theta_2 \partial^\alpha \delta \theta.
\end{split}
\end{equation*}

\paragraph{Estimating the terms $I_\gamma$}
\text{ }

We first observe that by integration by parts, $I_{0,0} = 0$ since $\nabla \cdot u_1 = 0$. For $\gamma \neq (0,0)$, we have that $|\alpha - \gamma| \le \ell-1$ so
\begin{equation*}
    |I_\gamma| \le \| \partial^\gamma u_1\|_{L^\infty} \|\partial^{\alpha - \gamma} \nabla \delta \theta \|_{L^2} \|\partial^\alpha \delta \theta \|_{L^2} \le \| \partial^\gamma u_1\|_{L^\infty} \| \delta \theta \|_{H^\ell}^2.
\end{equation*}
Using that $u_1 = v_1 + H_{1,\rho}$, 
\begin{equation*}
    \| \partial^\gamma u_1\|_{L^\infty} \le \| \partial^\gamma v_1 \|_{L^\infty} + \| \partial^\gamma H_{1,\rho} \|_{L^\infty}.
\end{equation*}
The term $\|\partial^\gamma H_{1,\rho} \|_{L^\infty} < +\infty$ is a constant that depends only on $\gamma$ and $\rho$ (with our notations, $\|\partial^\gamma H_{1,\rho} \|_{L^\infty} \lesssim 1$). We then use Proposition~\ref{p:FourierMultiplierKS} to get that there exists $\eps > 0$ small enough such that
\begin{equation*}
    \begin{split}
        \| \partial^\gamma v_1 \|_{L^\infty} & \le \| \triangle_{-1} \partial^\gamma v_1 \|_{L^\infty} + \| (\mathrm{Id} - \triangle_{-1})\partial^\gamma v_1 \|_{L^\infty} \\
        & \lesssim \| \theta_1 \|_{H^k} + \| (\mathrm{Id} - \triangle_{-1})\partial^\gamma v_1 \|_{H^{1+\eps}} \\
        & \lesssim \| \theta_1 \|_{H^k} + \| \partial^\gamma \theta_1 \|_{H^{1+\eps+1-2s}} \\
        & \lesssim  \| \theta_1 \|_{H^k} + \| \theta_1 \|_{H^{\ell+2-2s+\eps}} \\
        & \lesssim \| \theta_1 \|_{H^k},
    \end{split}
\end{equation*}
where the last inequality comes from the hypothesis that $\ell +2 \le k$. In conclusion,
\begin{equation*}
     \| \partial^\gamma v_1 \|_{L^\infty} \lesssim \| \theta_1 \|_{H^k} \lesssim \| \theta_1(0) \|_{H^k} \lesssim 1,
\end{equation*}
so that for every $\gamma \le \alpha$,
\begin{equation}\label{eq:est_I_gamma}
    |I_\gamma| \lesssim \|\delta\theta\|_{H^\ell}^2.
\end{equation}

\paragraph{Estimating the terms $J_\gamma$}
\text{ }

We start by the case $\gamma \neq \alpha$. We have that
\begin{equation*}
    J_\gamma = \int \big(\triangle_{-1} \partial^\gamma \delta v \big) \cdot \partial^{\alpha - \gamma} \nabla \theta_2 \partial^\alpha \delta \theta + \int \big(\mathrm{Id} -\triangle_{-1}\big)\partial^\gamma \delta v \cdot \partial^{\alpha - \gamma} \nabla \theta_2 \partial^\alpha \delta \theta
\end{equation*}
The first term satisfies, using Proposition~\ref{p:FourierMultiplierKS}, that
\begin{equation*}
\begin{split}
    \left| \int \big(\triangle_{-1} \partial^\gamma \delta v \big) \cdot \partial^{\alpha - \gamma} \nabla \theta_2 \partial^\alpha \delta \theta \right| & \le \| \triangle_{-1} \partial^\gamma \delta v \|_{L^\infty}\| \partial^{\alpha-\gamma} \nabla \theta_2 \|_{L^2} \|\delta\theta \|_{H^\ell} \\
    & \lesssim \| \partial^\gamma \delta \theta \|_{L^2} \|\theta_2\|_{H^{\ell+1}}\|\delta\theta \|_{H^\ell} \\
    & \lesssim \|\delta\theta \|_{H^\ell}^2,
\end{split}
\end{equation*}
where we used that $\ell \ge 2 > 2-2s$.
For the second term, if $\gamma \neq (0,0)$ we have for $\eps \in (0,1)$ that
\begin{equation*}
\begin{split}
     \left| \int \big(\mathrm{Id} -\triangle_{-1}\big)\partial^\gamma \delta v \cdot \partial^{\alpha - \gamma} \nabla \theta_2 \partial^\alpha \delta \theta\right|
     & \le \| (\mathrm{Id} - \triangle_{-1}) \partial^\gamma \delta v \|_{L^2} \| \partial^{\alpha-\gamma} \nabla \theta_2 \|_{L^\infty} \|\partial^\alpha \delta\theta \|_{L^2} \\
     & \lesssim  \|\partial^\gamma \delta \theta \|_{H^{1-2s}} \| \partial^{\alpha-\gamma} \nabla \theta_2 \|_{H^{1+\eps}} \|\delta\theta \|_{H^\ell} \\
     & \lesssim \| \delta\theta \|_{H^{1-2s+\ell-1}} \| \theta_2 \|_{H^{\ell-1+1+\eps}}\|\delta\theta \|_{H^\ell} \\
     & \lesssim \|\delta\theta \|_{H^\ell}^2,
\end{split}
\end{equation*}
where we used that $\gamma \neq \alpha$ so that $|\gamma| \le \ell -1$. In the case $\gamma = (0,0)$, we have that
\begin{equation*}
\begin{split}
    \left| \int \big(\mathrm{Id} -\triangle_{-1}\big) \delta v \cdot \partial^{\alpha} \nabla \theta_2 \partial^\alpha \delta \theta\right| & \le \| (\mathrm{Id} - \triangle_{-1}) \delta v \|_{L^\infty} \| \partial^{\alpha} \nabla \theta_2 \|_{L^2} \|\partial^\alpha \delta\theta \|_{L^2} \\
    & \lesssim \| (\mathrm{I_d} - \triangle_{-1}) \delta v \|_{H^{1+\eps}} \| \theta_2 \|_{H^{\ell +1}} \|\delta\theta \|_{H^\ell} \\
    & \lesssim \| \delta\theta\|_{H^{1+\eps+1-2s}} \|\delta\theta \|_{H^\ell} \\
    & \lesssim \|\delta\theta \|_{H^\ell}^2. 
\end{split}
\end{equation*}
We now estimate $|J_\alpha|$. As in Section~\ref{sec:apriori}, let us denote by $A = (A_1,A_2) =  \Big( - \partial_2(-\Delta)^{-s},\partial_1(-\Delta)^{-s}\Big)$, which is a (formally) skew-symmetric operator. We have that
\begin{equation*}
    \begin{split}
        J_\alpha & = -\sum_{j=1}^2 \int \partial^\alpha A_j (\delta \theta) \partial_j \theta_2 \partial^\alpha \delta\theta \\ 
        & = -\sum_{j=1}^2 \int \partial^\alpha \delta \theta A_j  (\partial_j \theta_2 \partial^\alpha \delta\theta) 
    \end{split}
\end{equation*}
and by summing both of these last relations we obtain that
\begin{equation*}
    J_\alpha = \frac{1}{2} \int \partial^\alpha \delta \theta \big[ A_j,\delta_j \theta_2 \big] (\partial^\alpha \delta \theta).
\end{equation*}
We now use Lemma~\ref{p:CCCGW_Commutator} to get that
\begin{equation*}
    |J_\alpha| \lesssim \left(\| \partial^\alpha \delta\theta \|_{L^2} \left\| \mathcal{F}\left[(- \Delta)^{(1 - 2s)/2} \delta_j \theta_2\right] \right\|_{L^1} + \| (- \Delta)^{s} \partial^\alpha\delta \theta \|_{L^2} \left\| \mathcal{F}\left[(- \Delta)^{1/2} \delta_j \theta_2\right] \right\|_{L^1} \right) \|\partial^\alpha\delta \theta\|_{L^2}.
\end{equation*}
We now compute each of these terms. Using the Cauchy-Schwarz inequality, we have for any $\eps >0$ that
\begin{equation*}
\begin{split}
    \left\| \mathcal{F}\left[(- \Delta)^{1/2} \delta_j \theta_2\right] \right\|_{L^1} & \le \int |\xi|^2 |\what{\theta_2}(\xi)| \dd \xi \\
    & = \int \frac{|\xi|^2}{(1+|\xi|)^{3+\eps}}(1+|\xi|)^{3+\eps}|\what{\theta_2}(\xi)| \dd \xi \\
    & \le \left(\int \frac{|\xi|^4}{(1+|\xi|)^{6+2\eps}} \dd \xi\right)^{1/2} \|\theta_2\|_{H^{3+\eps}} \\
    & \lesssim 1
\end{split}
\end{equation*}
and
\begin{equation*}
\begin{split}
    \left\| \mathcal{F}\left[(- \Delta)^{(1-2s)/2} \delta_j \theta_2\right] \right\|_{L^1} & \le \int |\xi|^{2-2s} |\what{\theta_2}(\xi)| \dd \xi \\
    & \le \int (1+|\xi|)^2 |\what{\theta_2}(\xi)| \dd \xi \\
    & = \int \frac{(1+|\xi|)^2}{(1+|\xi|)^{3+\eps}}(1+|\xi|)^{3+\eps}|\what{\theta_2}(\xi)| \dd \xi \\
    & \le \left(\int \frac{(1+|\xi|)^4}{(1+|\xi|)^{6+2\eps}} \dd \xi\right)^{1/2} \|\theta_2\|_{H^{3+\eps}} \\
    & \lesssim 1.
\end{split}
\end{equation*}
We compute now the final term using the fact that $\ell -2s \ge 0$,
\begin{equation*}
    \begin{split}
        \| (- \Delta)^{s} \partial^\alpha\delta \theta \|_{L^2} & = \left(\int \left(|\xi|^{-2s}|\xi|^{\ell} |\what{\delta\theta}(\xi)|\right)^2\dd \xi \right)^{1/2} \\
        & \le \left(\int (1+|\xi|^\ell)^2 |\what{\delta\theta}(\xi)|^2\dd \xi \right)^{1/2} \\
        & = \| \delta\theta\|_{H^\ell}.
    \end{split}
\end{equation*}
In conclusion, we have that
\begin{equation*}
    |J_\alpha| \lesssim \|\delta\theta\|_{H^\ell}^2,
\end{equation*}
and thus for every $\gamma \le \alpha$,
\begin{equation}\label{eq:est_J_gamma}
    |J_\gamma| \lesssim \|\delta\theta\|_{H^\ell}^2.
\end{equation}

\paragraph{Estimating the terms $L_\gamma$}
\text{ }

Similarly to the $L^2$ estimate, we have that
\begin{equation*}
\begin{split}
    |L_\gamma| & \le \| \partial^\gamma \delta H_{\rho} \|_{L^\infty} \| \partial^{\alpha-\gamma} \nabla \theta_2 \|_{L^2} \| \partial^\alpha \delta\theta \|_{L^2} \\
    & \lesssim \| \partial^\gamma \delta H_{\rho} \|_{L^\infty} \|\delta \theta \|_{H^\ell}.
    \end{split}
\end{equation*}
Observing that $\partial^\gamma \delta H_{\rho}$ is a smooth map, we have that
\begin{equation*}
    \|\partial^\gamma \delta H_{\rho} \|_{L^\infty} \le \| \nabla \partial^\gamma K_{s,\rho} \|_{L^\infty} |\delta z| \lesssim |\delta z|
\end{equation*}
and thus in the end,
\begin{equation}\label{eq:est_L_gamma}
    |L_\gamma|  \lesssim |\delta z| \|\delta \theta \|_{H^\ell}.
\end{equation}

\paragraph{Conclusion}
\text{ }

Gathering estimates~\eqref{eq:est_I_gamma},~\eqref{eq:est_J_gamma} and~\eqref{eq:est_L_gamma} summed over every $|\alpha|=\ell$ and plugging them into relation~\eqref{eq:Senna}, we obtain that
\begin{equation*}
     \der{}{t} \| \partial^\alpha \delta \theta \|_{L^2}^2 \lesssim \|\delta\theta\|_{H^\ell} \Big(\|\delta\theta\|_{H^\ell} + |\delta z| \Big).
\end{equation*}
Adding this with relation~\eqref{eq:Prost_v2} we get that
\begin{equation*}
    \der{}{t} \| \delta\theta \|_{H^\ell}^2 \lesssim \|\delta \theta\|_{L^2}\big(|\delta z| + \| \delta \theta\|_{H^{1-2s}}\big)  + \|\delta\theta\|_{H^\ell} \Big(\|\delta\theta\|_{H^\ell} + |\delta z| \Big) \lesssim \|\delta\theta\|_{H^\ell} \Big(\|\delta\theta\|_{H^\ell} + |\delta z| \Big),
\end{equation*}
which yields that
\begin{equation*}
    \der{}{t} \| \delta\theta \|_{H^\ell} \lesssim |\delta z| + \| \delta\theta \|_{H^\ell},
\end{equation*}
which concludes the proof of Lemma~\ref{lem:der_theta_2-theta_1}.
\end{proof}

\subsection{Conclusion of the proof of Proposition~\ref{prop:stabilité} and uniqueness in Theorem~\ref{thrm:strong_solutions}}\label{sec:conclusion_uniqueness}

We now conclude the proof of Proposition~\ref{prop:stabilité}.

Gathering the results of Lemma~\ref{lem:choice_rho_T}, Lemma~\ref{lem:der_z_2-z_1} and Lemma~\ref{lem:der_theta_2-theta_1}, we obtain that there exists a time $T_0$ such that for every $t \in [0,T_0]$,
\begin{equation*}
    \der{}{t} \big( \| \delta \theta \|_{H^\ell} + |\delta z|\big) \le  C\big(\|\delta\theta\|_{H_\ell} + |\delta z|\big),
\end{equation*}
where $T_0$ and the constant $C$ depend only on $\|\theta_1(0,\cdot)\|_{H^k}$, $\|\theta_2(0,\cdot)\|_{H^k}$, $R_1$, $R_2$, $\ell$ and $T$.

By the classical Gronwall's Lemma, we infer that
\begin{equation*}
    \| \delta \theta \|_{H^\ell} + |\delta z| \le\Big( \| \delta \theta(0,\cdot) \|_{H^\ell} + |\delta z(0)| \Big) e^{Ct}.
\end{equation*}
This concludes the proof of Proposition~\ref{prop:stabilité}, and hence of Theorem~\ref{thrm:strong_solutions}-$(ii)$. \qed

\medskip

Please note that in the case $s \ge 1/2$ replacing the result of Lemma~\ref{lem:der_theta_2-theta_1} by relation~\eqref{eq:pour_stabilité_L2}, then by the same argument using Gronwall's Lemma, we prove Proposition~\ref{prop:stabilité_L2}.

\medskip

We now prove the uniqueness of strong solutions to the vortex wave system~\eqref{eq:StrongSystem} as stated in Theorem~\ref{thrm:strong_solutions}-$(i)$, and whose existence has been proved in Section~\ref{sec:existence_strong}. Let $(\theta_1,z_1)$ and $(\theta_2,z_2)$ be two solutions with same initial datum $(\theta_0,z_0)$, with
\begin{equation*}
    \theta_0 \equiv \beta \quad \text{ on } B(z_0,R_0).
\end{equation*}
Assume that both solutions are defined on $[0,T)$. We then define $\bar{T}$ the supremum of times such that the two solutions coincide. If $\bar{T} < T$, starting from time $\bar{T}$, those solutions clearly satisfy the hypotheses of Proposition~\ref{prop:stabilité} and satisfy the existence of a time $T_0\in(\bar{T}, T)$ such that for every $t \in [\bar{T},T_0]$,
\begin{equation*}
    \| \delta\theta(t) \|_{H^1} + |\delta z(t)| \lesssim \| \delta\theta(\bar{T}) \|_{H^1} + |\delta z(\bar{T})| = 0,
\end{equation*}
which is absurd by definition of $\bar{T}$. This proves that $\bar{T} = T$ and therefore that maximal strong solutions are unique. \qed

\section{Blow-up criterion}\label{sec:blow-up}

A first obstacle to obtain global existence of strong solutions to the vortex-wave system is that the active scalar, whose self-interaction is governed by the SQG equations, can blow-up on his own and lose regularity. Although we believe that a proof that this blow-up can actually happen is not known, there are numerical evidences suggesting it~\cite{Scott_Dritschel_2014}. Moreover, a second type of blow-up, specific to the vortex-wave system, may happen. Indeed, the proof of existence of solutions requires that the active scalar should be locally constant around the point-vortex. If the distance to the non -constant part $t\mapsto R(t)$ goes to 0 in finite time, the equations become singular, resulting in a loss of regularity.
In this article we investigate this second type of blow-up.

\subsection{The case where $R(t)$ vanishes : proof of Theorem~\ref{thrm:blow-up}}
 For any $f \in C^0(\R^2)$, $z \in \R^2$ and $R \in (0,1)$, we introduce the quantity $\cN (f,z,R) \in (-\infty;+\infty]$ given by
\begin{equation*}
    \cN(f,z,R) := \max_{\{x \in \R^2, \, |x-z| = R\}} \frac{-(x-z)\cdot\big(f(x)-f(z)\big)}{R^2(1-\ln R)},
\end{equation*}
which evaluates the behavior of the radial component of $f(x)-f(z)$ near $z$. This quantity is a directional Log-Lipschitz norm. In particular, it satisfies that
\begin{equation*}
    |\cN(f,z,R)| \le \| f \|_{LL} := \sup_{0<|x-y|<1} \frac{|f(x)-f(y)|}{|x-y|(1-\ln|x-y|)}.
\end{equation*}
We prove the following blow-up criterion.
\begin{proposition}[Criterion for the shrink of the constant part]\label{prop:blow-up}
Let $(\theta,z)$ be a strong solution of \eqref{eq:StrongSystem} on $[0,T^\ast)$ with $T^* < \infty$, such that
\begin{equation*}
    \theta_0 \equiv \beta \qquad \text{ on } B\big(z(0),R_0\big),
\end{equation*}
for some $\beta \in \RR$. We recall the definition of $t \mapsto R(t)$ given in relation~\eqref{def:R(t)}. 

Then,
    \begin{equation*}
        R(t) \tend{t}{T^\ast} 0 \; \Longleftrightarrow \; \int_0^{T^\ast} \cN\big(v(t,\cdot),z(t),R(t)\big) \dd t = +\infty.
    \end{equation*}
\end{proposition}

\begin{proof}
We start similarly as in the proof of Lemma~\ref{lem:remain_constant_near_PV}. Let $x \in \R^2$ such that $\theta(t,x) \neq \beta$. Let $\tau \mapsto X(\tau)$ be the trajectory of the fluid particle passing in $x$ at time $t$. Then
\begin{equation*}\begin{split}
    \left. \der{}{\tau}|X(\tau) - z(\tau)|\right|_{\tau = t} & = \frac{x-z(t)}{\big|x-z(t)\big|} \cdot \big(v(x,t) + H(x,t) - v(z(t),t)\big)\\
    & = \frac{x-z(t)}{\big|x-z(t)\big|} \cdot \big(v(x,t) - v(z(t),t)\big)\\
    & \ge \min_{\{y \in \R^2, \, |y-z(t)| = R(t)\}} \frac{y-z(t)}{\big|y-z(t)\big|} \cdot \big(v(y,t) - v(z(t),t)\big) \\
    & = -\max_{\{y \in \R^2, \, |y-z(t)| = R(t)\}} -\frac{y-z(t)}{\big|y-z(t)\big|} \cdot \big(v(y,t) - v(z(t),t)\big)
\end{split}\end{equation*}
since $\big(x-z(t)\big) \cdot H(x,t)= 0$.
Since $v \in C^1([0,T^\ast))$, the quantity $\cN\Big(v(t,\cdot),z(t),R(t)\Big)$ is well defined and finite by assuming without lost of generality that $R(0) < 1$. We then obtain that $R$ satisfies the differential inequality
\begin{equation*}
   \der{}{t} R(t) \ge  -\cN\Big(v(t,\cdot),z(t),R(t)\Big)\, R(t)(1-\log(R(t)),
\end{equation*}
from which we infer, as a particular case of Osgood's uniqueness theorem that for every $t \in [0,T^\ast)$,
\begin{equation*}
    R(t) \ge \exp \left( - \exp\left( \int_0^t \cN\big(v(\tau,\cdot),z(\tau),R(\tau)\big)\dd \tau + \ln\big(1-\ln(R(0))\big) \right)+1\right).
\end{equation*}
and thus 
\begin{equation*}
            R(t) \tend{t}{T^*} 0 \Longleftrightarrow \int_0^{T^\ast} \cN\big(v(t,\cdot),z(t),R(t)\big) \dd t = +\infty.
\end{equation*}
\end{proof}
Theorem~\ref{thrm:blow-up} is a simple reformulation of Proposition~\ref{prop:blow-up} with $N(t) := \cN\big(v(t,\cdot),z(t),R(t)\big)$.

\subsection{Proof of Theorem~\ref{thrm:strong_solutions}-$(iii)$ and remarks}
From Proposition~\ref{prop:blow-up}, we obtain immediately Theorem~\ref{thrm:strong_solutions}-$(iii)$ from the observation that
\begin{equation*}
    N(t) \le \| v(t)\|_{LL} \lesssim \| \theta(t) \|_{H^{3-2s}}.
\end{equation*}
Let us mention that a naive approach to obtain a blow-up criterion could be to conclude directly from Lemma~\ref{lem:remain_constant_near_PV} that
\begin{equation*}
    R(t) \tend{t}{T^*} 0 \; \Longrightarrow \; \int_0^{T^\ast} \| \nabla v(t,\cdot) \|_{L^\infty} \dd t = +\infty.
\end{equation*}
From this, we only infer that for every $\eps > 0$,
\begin{equation*}
    \int_0^{T^\ast} \| \theta(t,\cdot) \|_{H^{3-2s + \eps}} \dd t = +\infty,
\end{equation*}
which is a weaker result than Theorem~\ref{thrm:strong_solutions}-$(iii)$. Moreover, the quantity $t\mapsto N(t)$ is much more precise than $\| \nabla v(t) \|_{L^\infty}$ or even than $\| v(t)\|_{LL}$. In particular, let us observe that since $\big(x-z(t)\big) \cdot H(t,x) = 0$, the point-vortex seems to play no role into the possible vanishing of $R$. This is not a rigorous assertion at all since $\theta$ is affected by the velocity field $H$ generated by the point-vortex. This make us think that should such a blow-up exist, then it might come from the same mechanism than the blow-up of the SQG equations alone.

\section{Global Existence of weak solutions}

\subsection{Proof of Theorem~\ref{thrm:global weak sub-critical}}\label{sec:proof-subcritical}

Let $\varepsilon>0$.
To start with, we consider the functions $\theta_\varepsilon$, $z_\varepsilon$, $v_\varepsilon$ and $H_\varepsilon$ global solution to the dynamic associated to the regularized kernel $K_{s,\varepsilon}$ (see Section~\ref{sec:reg} for details). 
Using the preservation of the $L^p$ norms by the flow stated by~\eqref{conservation law epsilon}, we have that $(\theta_\varepsilon)_{\varepsilon>0}$ is bounded in $L^\infty(\RR_+\times\RR^2)$. 
Therefore, by the Banach-Alaoglu theorem, we have $\theta_\varepsilon$ that converges weakly-$\star$ in $L^\infty(\RR_+\times\RR^2)$ up to an omitted extraction. 

The right-hand side of~\eqref{bio} is (up to a multiplicative constant) the Green function of the fractional Laplace operator $(-\Delta)^\sigma$ with $\sigma=s-1/2$. 
Note that we have $\sigma>0$ by hypothesis.
Therefore, using properties of homogeneous Fourier multipliers on Besov spaces (see~\cite[\S 2.3]{Bahouri_Chemin_Danchin_2011}), we have $v(t,.)\in B^{\sigma}_{p,p}(\RR^2)$. More precisely, there exists a constant $C>0$ such that
\begin{equation}\label{pomme}
    \forall\,t>0,\qquad\|v_\varepsilon(t,.)\|_{B^{\sigma}_{p,p}}\;\leq C_T\|\theta_0\|_{L^p}
\end{equation} for all $p\in[2,+\infty)$. 
By Sobolev inequalities, we have for $p$ large enough that $B^{\sigma}_{p,p}(\RR^2)$ is continuously embedded in the space of Hölder continuous functions $\cC^{0,\sigma'}(\RR^2)=:B^{\sigma'}_{\infty,\infty}$ with $0<\sigma'<\sigma$. 
With~\eqref{bio}, we have that $K_{s,\varepsilon}$ converges strongly in $L^1_{\rm loc}(\RR^2)$ by the Lebesgue dominated convergence theorem.
This gives as $\varepsilon\to0$:
\begin{equation*}
    v_\varepsilon:=K_{s,\varepsilon}\star\theta_\varepsilon\;\longrightarrow \;v:=K_{s}\star\theta,\qquad\text{strongly in } L^1_{\rm loc}\cap L^\infty(\RR^+\times\RR^2).
\end{equation*}
Thus, for any $\psi\in\cD([0,T)\times\RR^2)$ with $T>0$ fixed,
\begin{equation}\label{lego 1}
    \int_0^T\int_{\RR^2}\theta_\varepsilon(t,x)\;v_\varepsilon(t,x)\cdot\nabla\psi(t,x)\,\dd x\,\dd t\;\longrightarrow \;\int_0^T\int_{\RR^2}\theta(t,x)\;v(t,x)\cdot\nabla\psi(t,x)\,\dd x\,\dd t.
\end{equation}
The weak convergence of $\theta_\varepsilon$ as $\varepsilon\to0$ implies also
\begin{equation}\label{lego 2}
    \int_0^T\int_{\RR^2}\theta_\varepsilon(t,x)\,\frac{\partial\psi}{\partial t}(t,x)\,\dd x\,\dd t\;\longrightarrow \;\int_0^T\int_{\RR^2}\theta(t,x)\,\frac{\partial\psi}{\partial t}(t,x)\,\dd x\,\dd t.
\end{equation}

Similarly as before we have $v\in\cC^{0,\sigma'}$. 
We can then define $t\mapsto z(t)$ using~\eqref{eq:Euler Vortex Wave weak 2}. 
The uniform convergence of $v_\varepsilon$ towards $v$ implies the uniform convergence of $z_\varepsilon$ towards $z$. 
We observe that, since $v$ is bounded, we have $z\in\cC^{0,1}([0,T);\RR^2).$
This is enough to get the convergence:
\begin{equation*}
    H_\varepsilon(t,x):=K_{s,\varepsilon}(x-z_\varepsilon(t))\;\longrightarrow\;H(t,x):=K_{s}(x-z(t)),\qquad\text{strongly in }L^1_{\rm loc}([0,T)\times\RR^2).
\end{equation*}
Thus, 
\begin{equation}\label{lego 3}
    \int_0^T\int_{\RR^2}\theta_\varepsilon(t,x)\;H_\varepsilon(t,x)\cdot\nabla\psi(t,x)\,\dd x\,\dd t\;\longrightarrow\;\int_0^T\int_{\RR^2}\theta(t,x)\;H(t,x)\cdot\nabla\psi(t,x)\,\dd x\,\dd t
\end{equation}
Since $t\mapsto z(t)$ satisfies~\eqref{eq:Euler Vortex Wave weak 2} then with~\eqref{lego 1}\eqref{lego 2}\eqref{lego 3} we have the existence of $(\theta,z)$ weak solution to the SQG vortex-wave system~\eqref{eq:weak formulation}\eqref{eq:weak formulation 2}.

Finally, since the weak convergence decreases the $L^p$ norms, we have for all $p\in[1,+\infty]$ that $\|\theta(t,.)\|_{L^p}\leq\|\theta_\varepsilon(t,.)\|_{L^p}= \|\theta_0\|_{L^p},$ where the last equality is given by~\eqref{conservation law epsilon}.

The only point that remains to be proved is the continuity in time of $t\mapsto\theta(t,.)$ with values in $L^1\cap L^\infty(\RR^2)$ endowed with the weak-$\star$ topology. For $t_0,t\in\RR_+$ and for $\varphi$ a test function on $\RR^2$:
\begin{equation*}
    \big|\left<\theta(t,.)-\theta(t_0,.),\varphi\right>\big|\;\leq\;\big|\left<\theta(t,.)-\theta_\varepsilon(t,.),\varphi\right>\big|+\big|\left<\theta_\varepsilon(t,.)-\theta_\varepsilon(t_0,.),\varphi\right>\big|+\big|\left<\theta_\varepsilon(t_0,.)-\theta(t_0,.),\varphi\right>\big|.
\end{equation*}
We now remark that the estimate~\eqref{ordre et beaute} still holds independently on $N$ using the same reasoning from the evolution equation. As a consequence, we have:
\begin{equation*}
    \big|\left<\theta_\varepsilon(t_0,.)-\theta_\varepsilon(t,.),\varphi\right>\big|\;=\;\bigg|\int_{t_0}^t\int_{\RR^2}\frac{\partial\theta_{\varepsilon}}{\partial t}(s,x)\,\varphi(x)\,dx\,ds\bigg|\leq(t-t_0)\|\varphi\|_{H^1}\Big(\|K_{s,\varepsilon}\|_{L^\infty}+\|K_{s,\varepsilon}\|_{L^2}\|\theta_\varepsilon\|_{L^2}\Big)\|\theta_\varepsilon\|_{L^2}.
\end{equation*}
For fixed values of $\varepsilon$, this term vanishes as $t\to t_0$. We also have that $\big|\left<\theta(t,.)-\theta_\varepsilon(t,.),\varphi\right>\big|$ vanishes as $\varepsilon\to 0$ uniformly with respect to $t\in[0,T]$ as a consequence of~\eqref{lego 2}. 
Thus, $\left<\theta(t,.)-\theta(t_0,.),\varphi\right>\to0$ as $t\to t_0$

\qed

\subsection{Proof of Theorem~\ref{thrm:VFaible}}\label{sec:proofVfaible}
For the proof of the global existence of $V$-weak solutions to the vortex-wave system when $s=1/2$, ie: Theorem~\ref{thrm:VFaible}-$(i)$ et $(ii)$, we follow closely the construction by Marchand~\cite{Marchand_2008}. 
This allows us to make use of the technical lemmas previously in~\cite{Marchand_2008} and ease the reading of this proof and focus on what is new.
Indeed, the main novelties of this proof lay in the convergence of the term of interaction between the ODE part and the PDE part.

\subsubsection{Step 1: regularization of the system}
To start with, we consider the regularization of the kernel $K_{\frac{1}{2},\varepsilon}$ introduced at~\eqref{bio}. By Proposition~\ref{prop:global smooth approx sol}, the vortex-wave system associated to this regularized kernel admits a global smooth solution $(\theta_\varepsilon, z_\varepsilon)$ with the respective associated velocity fields $v_\varepsilon$ and $ H_\varepsilon$.

Since the total velocity field ($v_\varepsilon+H_\varepsilon$) is divergence-free, the flow preserves all the $L^p$ norm of functions.
As a consequence, by the Banach-Alaoglu theorem, we have weak convergence of $\theta_\varepsilon$ towards some function $\theta\in L^\infty(\RR_+;L^1\cap L^\infty(\RR^2))$ (up to an omitted extraction of subsequence) as $\varepsilon\to0$. By property of the weak convergence:
\begin{equation}\label{eq:decrease Lp star}
    \forall\,t>0,\qquad\|\theta(t,.)\|_{L^p}\;\leq\;\liminf\limits_{\varepsilon\to0}\|\theta_\varepsilon(t,.)\|_{L^p}\;\leq\;\limsup\limits_{\varepsilon\to0}\|\theta_\varepsilon(0,.)\|_{L^p}\;=\;\|\theta(0,.)\|_{L^p}.
\end{equation}
We also have for all $T>0$ and $\psi\in\cD([0,T)\times\RR^2)$,
\begin{equation*}
    \int_0^T\int_{\RR^2}\frac{\partial\psi}{\partial t}(t,x)\,\theta_\varepsilon(t,x)\,\dd x\,\dd t\;\longrightarrow\;\int_0^T\int_{\RR^2}\frac{\partial\psi}{\partial t}(t,x)\,\theta(t,x)\,\dd x\,\dd t.
\end{equation*}
Similarly as~\eqref{pomme}, a simple triangular inequality give the continuity in time with value in $L^\infty\cap L^1$ endowed with the weak-$\star$ topology.

\medskip

To conclude the proof of Theorem~\ref{thrm:VFaible}-$(i)$, we study the convergence of the other terms as $\varepsilon\to 0$.

\subsubsection{Step 2: convergence of the non-linear term}
We now study the non-linear term using the commutator formulation~\eqref{eq:commutator formulation}. We introduce the following notation for the associated bilinear form:
\begin{equation*}
    \cI_\psi(\theta,\vartheta):=\int_{\RR^2}K_\frac{1}{2}\star\theta\cdot\Big[(-\Delta)^\frac{1}{2},\;\nabla\psi\Big](-\Delta)^{-\frac{1}{2}}\;\vartheta(x)\,\dd x
\end{equation*}
To study the convergence of this term, we make use of the following technical lemma which are proved in~\cite{Marchand_2008}. First, we use the following estimate on the high-frequencies self interactions:
\begin{lemma}[Proposition 2.3 in~\cite{Marchand_2008}]\label{lem:L4/3 regularity}
    Let $\frac{4}{3}<p<2$.
    For $\theta\in L^{p}$, the distribution $K_s\star\theta\cdot\theta$ is well-defined using~\eqref{eq:commutator formulation} as a consequence of the following estimate on the high-frequencies:
    \begin{equation*}
        \Big|\cI_\psi\big(\mathrm{H}_j\theta, \mathrm{H}_j\theta\big)\Big|\;\leq\; C\,2^{-j\left(3-\frac{4}{p}\right)}\|\nabla^2\psi\|_{L^\infty}\,\|\mathrm{H}_j\theta\|_{L^p}^2,
    \end{equation*}
\end{lemma}
The proof of this result relies on Littlewood-Paley decomposition, the commutator estimates given by Lemma~\ref{lem:calderon} and Sobolev embeddings. Note that a consequence of this lemma is that $\cI(\theta,\theta)$ is well-defined for all $\theta\in L^p$ with $4/3<p<2$. \vspace{0.2cm}

We also need the following compactness result for the low frequencies of $\theta$:
\begin{lemma}[Lemma~9.2 in~\cite{Marchand_2008}]\label{lem:strong compactness}
    Let $(\theta_n)_{n\in\NN}$ be a sequence bounded in $L^\infty(\RR_+;L^p(\RR^2))$ for $1<p<+\infty$. Then for every $j\in\ZZ$, the sequences $(\mathrm{S}_j\theta_n)_{n\in\NN}$ and $(K_\frac{1}{2}\star(\mathrm{S}_j\theta_n))_{n\in\NN}$ are pre-compact in $L^\frac{p}{p-1}(\RR_+\times\RR^2)$.
\end{lemma}
The proof of this lemma uses the smoothing properties of the operators $(S_j)$ and a compact embedding argument.\vspace{0.2cm}

\begin{lemma}Let $\theta_\varepsilon$ be a solution in the regularized vortex-wave system~\eqref{eq:ApproxSystem1}.
We have convergence of the non-linear term:
    $\cI_\psi(\theta_\varepsilon,\theta_\varepsilon)\;\longrightarrow\;\cI_\psi(\theta,\theta).$
\end{lemma}
\begin{proof}
we decompose the studied non-linear term between low and high frequencies (we fix $j\in\ZZ$):\begin{equation}\label{pie}
\cI_\psi(\theta_\varepsilon,\theta_\varepsilon)=\cI_\psi(\mathrm{S}_j\theta_\varepsilon,\mathrm{S}_j\theta_\varepsilon)+\cI_\psi(\mathrm{S}_j\theta_\varepsilon,\mathrm{H}_j\theta_\varepsilon)+\cI_\psi(\mathrm{H}_j\theta_\varepsilon,\mathrm{S}_j\theta_\varepsilon)+\cI_\psi(\mathrm{H}_j\theta_\varepsilon,\mathrm{H}_j\theta_\varepsilon).
\end{equation}
It is a direct consequence of Lemma~\ref{lem:strong compactness} that the three first integrals appearing in the decomposition~\eqref{pie} are converging as $\varepsilon\to 0$ towards their expected limit quantity:
\begin{equation*}\begin{split}
&\cI_\psi(\mathrm{S}_j\theta_\varepsilon,\mathrm{S}_j\theta_\varepsilon)\;\longrightarrow\;\cI(\mathrm{S}_j\theta,\mathrm{S}_j\theta),\\
&\cI_\psi(\mathrm{S}_j\theta_\varepsilon,\mathrm{H}_j\theta_\varepsilon)\;\longrightarrow\;\cI(\mathrm{S}_j\theta,\mathrm{H}_j\theta),\\
&\cI_\psi(\mathrm{H}_j\theta_\varepsilon,\mathrm{S}_j\theta_\varepsilon)\;\longrightarrow\;\cI(\mathrm{H}_j\theta,\mathrm{S}_j\theta).
\end{split}
\end{equation*}
Concerning the remaining term (ie: the high-frequency self interaction), we simply write, using Lemma~\ref{lem:L4/3 regularity} and Equation~\eqref{eq:decrease Lp star}:
\begin{equation*}
    \Big|\cI_\psi(\mathrm{H}_j\theta_\varepsilon,\mathrm{H}_j\theta_\varepsilon)\Big|+\Big|\cI_\psi(\mathrm{H}_j\theta,\mathrm{H}_j\theta)\Big|\;\leq\;C_0\,2^{-j},
\end{equation*}
where the constant $C_0$ above depends on the test function $\psi$ and on the initial datum $\theta_0$. We conclude
\begin{equation*}
    \limsup\limits_{\varepsilon\to 0}\Big|\cI_\psi(\theta_\varepsilon,\theta_\varepsilon)-\cI_\psi(\theta,\theta)\Big|\;\leq\;C_0\,2^{-j}
\end{equation*}
Since the estimate above holds for all $j\in\ZZ$, we infer as $\varepsilon\to 0$:
\begin{equation}\label{Helene}
    \cI_\psi(\theta_\varepsilon,\theta_\varepsilon)\;\longrightarrow\;\cI_\psi(\theta,\theta).
\end{equation}
\end{proof}

From the convergence~\eqref{Helene}, we can conclude that we have the expected convergence of the non-linear term:
\begin{equation*}
    \int_{\RR^2}K_{\frac{1}{2},\varepsilon}\star\theta_\varepsilon(x)\cdot\nabla\psi(x)\,\theta_\varepsilon(x)\,\dd x\;\longrightarrow\;\cI_\psi(\theta,\theta).
\end{equation*}

\subsubsection{Step 3: convergence of the ODE}
\begin{lemma}\label{lem:conv ODE}
Let $z_\varepsilon$ be the solution to the ODE in the regularized vortex-wave system~\eqref{eq:ApproxSystem1}. As $\varepsilon\to0$, it converges towards $z\in\cC^{0,1}(\RR_+;\RR^2)$ solution to 
    \begin{equation}\label{eq:Duhamel is here}
            z(t)\;=\;z_0+\int_0^t\int_{\RR^2}v(t,x)\,\chi\big(x-z(t)\big)\,\dd x.
    \end{equation}
\end{lemma}
\begin{proof}
To start with, we recall that
\begin{equation*}
    \der{z_\varepsilon}{t}(t)\;=\;\int_{\RR^2}v_\varepsilon(t,.)\,\chi\big(x-z_\varepsilon(t)\big)\,\dd x.
\end{equation*}
The kernel $\nabla^\perp(-\Delta)^{-\frac{1}{2}}$ is associated to the Fourier multiplier $\xi^\perp/|\xi|$, which is a bounded Fourier multiplier (see Lemma~\ref{lem:K 1/2}).
On the other hand, since $\theta_\varepsilon$ converges weakly towards $\theta$ then the Fourier transform $\widehat{\theta_\varepsilon}$ converges weakly towards $\widehat{\theta}$.
Therefore, we get that $v_\varepsilon(t,.):=K_{\frac{1}{2}}\star\theta_\varepsilon(t,.)$ converges weakly towards $v(t,.):=K_{\frac{1}{2}}\star\theta(t,.).$

On the other hand, it is a general property of functions in $W^{1,p}$ that
\begin{equation*}
    \Big\|\chi(.-z_2)-\chi(.-z_1)\Big\|_{L^p}\;\leq\;\|\nabla\chi\|_{L^p}\,|z_2-z_1|.
\end{equation*}
Therefore, fixing a $t>0$, we have for any sequence $\zeta_\varepsilon$ converging towards some $\zeta\in\RR^2,$
\begin{equation*}
    \int_{\RR^2}v_\varepsilon(t,.)\,\chi\big(x-\zeta_\varepsilon\big)\,\dd x\quad\longrightarrow\quad\int_{\RR^2}v(t,.)\,\chi\big(x-\zeta\big)\,\dd x
\end{equation*}
Then, by a bootstrap argument, we have convergence of $\dd z_\varepsilon/\dd t$ towards $\int_{\RR^2}v_\varepsilon(t,x)\,\chi\big(x-z_\varepsilon(t)\big)\,\dd x$ for every fixed $t>0$ and thus we have local uniform convergence of $z_\varepsilon$ towards a $z$ that satisfy~\eqref{eq:Duhamel is here}.
\end{proof}

\subsubsection{Convergence of the singular part of the velocity field}
To study the convergence of $H_\varepsilon$, the singular part of the velocity field (ie the velocity field generated by the point-vortex), it is necessary to exploit the fact that we only study $V$-weak solutions. 
In this parts and in the next part of the proof, we manipulate this notion of $V$-weak solutions. 
The objective is to show that this notion permits to obtain a non-trivial object in the presence of a localized singularity (in our case it is the point-vortex) even if the singularity is non-integrable.  

\medskip

It is now time to introduce the set $V$ we are going to work with:
\begin{lemma}
    Let $H_\varepsilon, z_\varepsilon$ defined by~\eqref{eq:ApproxSystem1} and let $t\mapsto z(t)$ the limit of $z_\varepsilon$ as $\varepsilon\to 0$ (see Lemma~\ref{lem:conv ODE}. Let $H(t,x):=K_\frac{1}{2}(x-z(t))$. Let $\delta>0$, we define
\begin{equation*}
    V_\delta\;:=\;\Big\{\psi\in\cD(\RR_+\times\RR^2)\;:\;\forall\;t>0,\quad\forall x\in\cB\big(z(t),\delta\big),\quad\nabla\psi(t,x)=0.\Big\}.
\end{equation*}
Then, for all $\psi\in V_\delta$, we have $\nabla\psi\cdot H\in L^1$ and for all $t>0$:
\begin{equation*}
\int_{\RR^2}\nabla\psi(t,x)\cdot H_\varepsilon(t,x)\,\dd x\;\longrightarrow\;\int_{\RR^2}\nabla\psi(t,x)\cdot H(t,x)\,\dd x.
\end{equation*}
\end{lemma}

\begin{proof}
The idea is to choose $V$ so that the most singular part (that lays in a moving point) vanishes. 
We recall the definition of the regularization kernel from Subsection \ref{sec:reg} where we introduced the functions 
\begin{equation*}
   {H}_\varepsilon(t,x)\;:=\;K_{\frac{1}{2},\varepsilon}(x-z_\varepsilon(t)),\qquad\text{and}\qquad \widetilde{H}_\varepsilon(t,x)\;:=\;K_{\frac{1}{2},\varepsilon}(x-z(t))
\end{equation*}
To start with, we have for all $\psi\in V_\delta$, on any fixed interval of time $[0,T)$,
\begin{equation*}
    \forall\,(t,x)\in[0,T)\times\RR^2,\qquad\widetilde{H}_\varepsilon(t,x)\cdot\nabla\psi(t,x)\;=\;H(t,x)\cdot\nabla\psi(t,x)
\end{equation*}
provided that $\varepsilon\leq\delta$. 
Indeed, the two functions $\widetilde{H}_\varepsilon(t,.)$ and $H(t,.)$ coincide outside $\cB(z(t),\delta)$ while $\nabla\psi(t,.)$ vanishes inside this ball.
On the other hand, if we assume that $\varepsilon$ is so small that for all $t\in[0,T],$ we have $|z_\varepsilon(t)-z(t)|\leq\delta/2$ and $\varepsilon\leq\delta/2,$ then, using again $\nabla\psi(t,.)\equiv0$ on $\cB(z(t),\delta)$,
\begin{equation*}\begin{split}
    \int_{\RR^2}\nabla\psi(t,x)&\cdot\Big(H_\varepsilon(t,x)-\widetilde{H}_\varepsilon(t,x)\Big)\dd x\;\\&=\;\int_{\RR^2\setminus\cB(z(t),\delta)}\nabla\psi(t,x)\cdot\Big(K_{\frac{1}{2},\varepsilon}(x-z_\varepsilon(t))-K_{\frac{1}{2},\varepsilon}(x-z(t))\Big)\dd x\\
    &=\;\int_{\RR^2\setminus\big(\cB(z(t),\frac{\delta}{2})\cup\cB(z_\varepsilon(t),\frac{\delta}{2})\big)}\nabla\psi(t,x)\cdot\Big(K_{\frac{1}{2},\varepsilon}(x-z_\varepsilon(t))-K_{\frac{1}{2},\varepsilon}(x-z(t))\Big)\dd x\\
    &=\;\int_{\RR^2}\nabla\psi(t,x)\cdot\Big(K_{\frac{1}{2},\frac{\delta}{2}}(x-z_\varepsilon(t))-K_{\frac{1}{2},\frac{\delta}{2}}(x-z(t))\Big)\dd x.
    \end{split}
\end{equation*}
Thus,
\begin{equation*}
    \bigg|\int_{\RR^2}\nabla\psi(t,x)\cdot\Big(H_\varepsilon(t,x)-\widetilde{H}_\varepsilon(t,x)\Big)\dd x\bigg|\;\leq\;\|\nabla\psi(t,.)\|_{L^2}\;\Big\|\nabla K_{\frac{1}{2},\frac{\delta}{2}}(t,.)\Big\|_{L^2}\;|z_\varepsilon(t)-z(t)|.
\end{equation*}
This implies the announced convergence.\end{proof}

Gathering all the convergences established before, we can conclude the proof of Theorem~\ref{thrm:VFaible}-$(i)$ by setting $V:=\bigcup_{\delta>0}V_\delta$.


\subsubsection{Conclusion of the proof of Theorem~\ref{thrm:VFaible}}

Concerning Theorem~\ref{thrm:VFaible}-$(ii)$, we already have for all $\psi\in V$ that $H\cdot\nabla\psi\in L^1$. 
For the $\Gamma$-coincident property, we  note that $V\subseteq\cD(\RR_+\times\RR^2)$. 
Conversely, let $\Omega_0\subseteq\RR_+\times\RR^2$ such that $\delta_0:=\inf_{x\in\Omega_0}\inf_{y\in\Gamma}|x-y|>0$. 
For a fixed $\varphi\in\cD(\RR_+\times\RR^2)$, it is possible to construct a $\psi\in V_{\delta_0/2}$ such that $\varphi\equiv\psi$ on $\Omega_0$ by the Urysohn's lemma. More precisely, let $\chi:\RR_+\to\RR_+$ monotonous, $\cC^\infty$ and such that $\chi(s)=1$ if $s\leq\delta_0/2$ and $\chi(s)=0$ if $s\geq\delta_0$. We define
\begin{equation}\label{def:psi delta}
    \psi_{\delta_0}(t,x)\;:=\;\chi\big(\|x-z(t)\|\big)-\!\!\!\!\!\!\!\int_{\cB\big(z(t),\frac{\delta_0}{2}\big)}\varphi(t,x)\,\dd x\;+\;\Big(1-\chi\big(\|x-z(t)\|\big)\Big)\,\varphi(t,x),
\end{equation}
where the barred integral denotes the mean value. It is direct to check that $\psi_{\delta_0}\in V_{\delta_0}\subset V$ and that $\psi_{\delta_0}\equiv\varphi$ in $\Omega_0.$

To prove Theorem~\ref{thrm:VFaible}-$(iii)$, we assume that a $V$-weak solution $(\theta,z)$ has regularity $\cC^1$ on an interval of time $[0,T)$ with $T\in(0,+\infty]$ and that $\theta(t,.)$ is constant on a neighborhood of $z(t)$ for all $t\in[0,T)$.
We prove here that, with such regularity, this $V$-weak solution is actually a solution in the classical sense.

To start with, by hypothesis, there exists $\delta_0>0$ such that $\nabla\theta(t,.)\equiv 0$ on the ball $\cB(z(t),\delta)$ for all $t\in[0,T)$. Let $0<\delta<\delta_0$, let $\varphi\in\cD([0,T)\times\RR^2)$ and let $\psi_\delta$ defined by~\eqref{def:psi delta}. The fact that $\theta$ is constant on $\cB(z(t),\delta)$ implies
\begin{equation}\label{Stroustrup}
    \forall\,t\in[0,T),\quad\forall\,x\in\RR^2,\qquad\theta(t,x)\nabla\varphi(t,x)\;=\;\theta(t,x)\nabla\psi_\delta(t,x).
\end{equation}
On the other hand, since $\psi_\delta\in V$, we use that $\theta$ is a $V$-weak solution to write
\begin{equation*}\begin{split}
    \int_0^T\int_{\RR^2}\Bigg(\frac{\partial\psi_\delta}{\partial t}(t,x)+v(t,x)&\cdot\Big[(-\Delta)^s,\nabla\psi_\delta(t,x)\Big](-\Delta)^{-s}\\&+H(t,x)\cdot\nabla\psi_\delta(t,x)\Bigg)\theta(t,x)\;\dd x\,\dd t+\int_{\RR^2}\omega_0(x)\,\psi_\delta(0,x)\,\dd x=0,
    \end{split}
\end{equation*}
Using Lemma~\ref{lem:Commutator formulation}, since $\theta$ is $\cC^1$, gives
\begin{equation*}
    \int_0^T\int_{\RR^2}\Bigg(\frac{\partial\psi_\delta}{\partial t}(t,x)+\Big(v(t,x)+H(t,x)\Big)\cdot\nabla\psi_\delta(t,x)\Bigg)\theta(t,x)\;\dd x\,\dd t+\int_{\RR^2}\omega_0(x)\,\psi_\delta(0,x)\,\dd x=0,
\end{equation*}
Using~\eqref{Stroustrup}, we can replace $\nabla\psi_\delta$ by $\nabla\varphi$ in the equality above. It then becomes possible to pass to the limit $\delta\to 0$ and get
\begin{equation*}
    \int_0^T\int_{\RR^2}\Bigg(\frac{\partial\varphi}{\partial t}(t,x)+\Big(v(t,x)+H(t,x)\Big)\cdot\nabla\varphi(t,x)\Bigg)\theta(t,x)\;\dd x\,\dd t+\int_{\RR^2}\omega_0(x)\,\varphi(0,x)\,\dd x=0,
\end{equation*}
Since $\theta$ is a $\cC^1$ function and since $\varphi$ is any function in $\cD([0,T)\times\RR^2)$, we conclude that $\theta$ satisfy point-wise
\begin{equation*}
    \frac{\partial\theta}{\partial t}(t,x)+\Big(v(t,x)+H(t,x)\Big)\cdot\nabla\theta(t,x)\;=\;0.\qquad\text{and}\qquad\theta(0,x)\;=\;\omega_0(x).
\end{equation*}
Similar manipulations give that $t\mapsto z(t)$ satisfy the quasi-point-vortex equation point-wise.
\qed

\section{The case of several vortices}\label{sec:multi_points}

In the previous sections, we studied for the sake of clarity the case of a single point-vortex interacting with the background. Naturally, the results of this paper can be extended to the case of several point-vortices at least in non degenerate cases. Although this is not the main topic of this paper and we will not provide a proof of this claim, let us discuss briefly how to define and deal with several point-vortices.

Consider several point-vortices $(z_i)_{1\le i \le N}$ of intensities $a_i \in \R^\ast$. The vortex-wave system then writes
\begin{equation}\label{eq:VW-SQG_N}
\left\{\begin{split}
    &\quad v(t,x)=K_s\star\theta(t,x),  \\
    &\quad H_i(t,x)=K_s\big(x-z_i(t)\big), \quad \forall i \in \{1,\ldots,N\},\\
    &\quad \frac{\partial\theta}{\partial t}+\nabla\cdot\left[\Big(v+\sum_{i=1}^N a_i H_i\Big)\theta\right]=0,\\
    &\quad\der{z_i}{t}(t)=v\big(z_i(t),t\big) + \sum_{j \neq i} a_j H_j(z_i(t),t), \quad \forall i \in \{1,\ldots,N\}.
\end{split}\right.
\end{equation}

A common method in the study of point-vortices used to reduce this problem to the case of a single point-vortex is to consider the influence of other point-vortices as an exterior field. In that case we consider the equations
\begin{equation*}
\left\{\begin{split}
    &\quad v(t,x)=K_s\star\theta(t,x),  \\
    &\quad H(t,x)=K_s\big(x-z(t)\big), \\
    &\quad \frac{\partial\theta}{\partial t}+\nabla\cdot\left[\Big(v+ a H + F\Big)\theta\right]=0,\\
    &\quad\der{z}{t}(t)=v\big(z_i(t),t\big) + F(z_i(t),t),
\end{split}\right.
\end{equation*}
with $F$ being a bounded, Lipschitz, divergence free map. It is then not difficult to reproduce the arguments of local existence and uniqueness of strong solutions for this new system. The remaining part consists in proving that the exterior field $F$ defined as the influence of the other point-vortices retains its properties locally in time. This is the case if the active scalar is locally constant around the point-vortices at time 0 and if the point-vortices remain far from each other.

The main difficulty arising when considering several point-vortices is that there exists a new type of blow-up: collapses of point-vortices. Assume that
\begin{equation*}
    \forall i \in \{1,\ldots,N\}, \quad \theta_0 \equiv \beta_i \quad \text{ on } B(z_i(0),R_i(0)),
\end{equation*}
for some constants $\beta_i$ and $R_i(0) > 0$ and let
\begin{equation*}
    R_i(t) := \sup \big\{ r \ge 0 \; , \; \theta(t) \equiv \beta_i \text{ on } B(z_i(t),r) \big\}.
\end{equation*}
If for some $T > 0$, 
\begin{equation}\label{eq:collapse}
    \lim_{t \to T} |z_i(t) - z_j(t)| = 0,
\end{equation}
which is what we call a collision of the point-vortices $z_i$ and $z_j$,
then if $\beta_i \neq \beta_j$ this implies that 
\begin{equation*}
    \lim_{t\to T} R_i(t) = 0 \quad \text{ and } \lim_{t\to T} R_j(t) = 0,
\end{equation*}
which is a more complicated version of the type of blow up that we already discussed in Section~\ref{sec:blow-up}, let us call it Type I blow-up.
However if $\beta_i = \beta_j$ it could be possible that \eqref{eq:collapse} holds true, while having that $R_i(t) > m$ and $R_j(t) > m$, with $m>0$, for every $t \in [0,T)$. This is a new type of blow-up, let us call it Type II, which is a classical point-vortex collapse in an exterior field that this time represents the interaction with the background. For references on the collisions of point-vortices in the SQG equations, we refer the reader to \cite{Reinaud_Dritschel_Scott_2022,Donati_Godard-Cadillac_2023_Holder_reg_collapse,Chen_Liu_2024} and references therein.

Both in the context of the Euler or SQG point-vortex system, if the intensities $a_i$ have all the same sign, then collapses can not happen and the solution is always global in time. For the Euler vortex-wave system, this remains true: there exists a global in time strong solution when the initial vorticity is locally constant around the point-vortices, see \cite{Marchioro_Pulvirenti_1991}, so that in that case, neither Type I nor Type II blow-up can happen. For the general SQG vortex-wave system this may not be true anymore because of the possible regularity blow-up of the background. More precisely, we prove that if all the intensities of the point-vortices are positive, then Type II blow-up can occur only if Type I occurs at the same time.
\begin{proposition}
Let $(\theta,(z_i)_i)$ be a classical (in the sense of Theorem~\ref{thrm:strong_solutions}-$(i)$) solution to the vortex-wave system \eqref{eq:VW-SQG_N} with $N$ points, on at least $[0,T^\ast)$. If all the intensities have the same sign, then for a collision of point-vortices to happen at time $T^\ast$ one must have that
\begin{equation*}
    \min_{1\le i \le N} R_i(t) \tend{t}{T^\ast} 0.
\end{equation*}

\end{proposition}

\begin{proof}
Assume without loss of generality that $a_i > 0$ for every $i \in \{1,\ldots,N\}$.
Let 
\begin{equation*}
    H(t) = \sum_{i \neq j} a_ia_j \frac{1}{\big|z_i(t)-z_j(t)\big|^{2-2s}}
\end{equation*}
and
\begin{equation*}
    I(t) = \sum_{i=1}^N a_i z_i^2(t).
\end{equation*}
Let $V(t) := \sup_{1\le i \le N} |v(t,z_i(t))|$. Then, recalling the anti-symmetry $u\cdot v^\perp=-v\cdot u^\perp$,
\begin{equation*}\begin{split}
    \der{}{t} I(t) & = 2 C_s \sum_{i=1}^N a_i z_i \cdot \left(\sum_{j \neq i} a_j\frac{\big( z_i(t)-z_j(t)\big)^\perp}{\big| z_i(t)-z_j(t)\big|^{4-2s}} + v(t,z_i(t)) \right) \\
    & = 2 C_s \sum_{i\neq j} a_ia_j \frac{z_i(t) \cdot z_j^\perp(t)}{\big| z_i(t)-z_j(t)\big|^{4-2s}} + 2 C_s \sum_{i=1}^N a_i z_i(t)\cdot v(t,z_i(t)) \\
    & = 2 C_s \sum_{i=1}^N a_i z_i(t)\cdot v(t,z_i(t)) \\
    & \le C (I(t) + 1)V(t),
\end{split}\end{equation*}
where we used that $z_i(t) \le z_i^2(t) +1$. Therefore, there exists a constant $C$ such that
\begin{equation*}
    \forall t \in [0,T^\ast), \quad I(t) \le C \exp\left(C\int_0^t V(t') \dd t'\right),
\end{equation*}
and thus
\begin{equation}\label{eq:control_I}
    \forall t \in [0,T^\ast), \quad \forall i \in \{1,\ldots,N\}, \quad |z_i(t)| \le C \sqrt{I(t)} \le  C \exp\left(C\int_0^t V(t') \dd t'\right).
\end{equation}
We then have that
\begin{equation*}\begin{split}
    \der{}{t} H(t)& = C_s \sum_{i=1}^N v(t,z_i(t)) \cdot \sum_{j \neq i} a_j\frac{\big( z_i(t)-z_j(t)\big)^\perp}{\big| z_i(t)-z_j(t)\big|^{4-2s}} \\
    & \le  C V(t) \exp\left(C\int_0^t V(t') \dd t'\right) \sum_{i\neq j} a_ia_j \frac{1}{\big|z_i(t)-z_j(t)\big|^{2-2s}} \\
    & \le  C V(t) \exp\left(C\int_0^t V(t') \dd t'\right) H(t)
\end{split}\end{equation*}
where we used relation \eqref{eq:control_I}. Therefore, 
\begin{equation*}
    \forall t \in [0,T^\ast), \quad H(t) \le C \exp\left( C \int_0^t V(\tau) \exp\left(C\int_0^\tau V(t') \dd t'\right) \dd \tau \right).
\end{equation*}
We now observe that
\begin{equation*}
    |V(t)| \le \frac{C}{\displaystyle \min_{1\le i \le N} \big(R_i(t)\big)^{3-2s}}
\end{equation*}
since, as in relation~\eqref{Tsunoda} and \eqref{eq:local_imply_0}, we have that
\begin{align*}
    v(t,z_i(t)) & = K_{s,R_i(t)}\star \theta(t,\cdot)(z_i(t)) \\ 
    & = \int_{|x-z_i(t)| \ge R_i(t)} K_s(x-z_i(t))\theta(t,x)\dd x \\
    & \le  \|\theta(t)\|_{L^1}\|K_s\|_{L^\infty(\R^2 \setminus B(z_i(t),R_i(t)))} \\
    & \le \frac{C_s}{\big(R_i(t)\big)^{3-2s}}.
\end{align*}
If a collision of point-vortices occurs at time $T^\ast$, then $H(t) \tend{t}{T^\ast} +\infty$ and thus gathering the previous relations, $V(t) \tend{t}{T^\ast} +\infty$ and
$\displaystyle \min_{1\le i \le N} R_i(t) \tend{t}{T^*} 0$.
\end{proof}

\paragraph{Acknowledgments.}
\text{ }
\medskip

Dimitri Cobb is supported by Deutsche Forschungsgemeinschaft (DFG, German Research Foundation) Project ID 211504053 - SFB 1060\medskip

Martin Donati's work has been supported by the Simons fundation project on Wave Turbulence, and the ENS de Lyon. \medskip

This work was supported by the BOURGEONS project, grant ANR-23-CE40-0014-01 of the French National Research Agency (ANR).

\end{document}